\documentclass[final]{siamart190516}
\usepackage{braket,amsfonts}
\usepackage{array}
\usepackage[caption=false]{subfig}
\usepackage{graphicx,epstopdf}
\usepackage{amsopn}
\usepackage{enumitem}
\usepackage{amsmath}
\usepackage{amsfonts}
\usepackage{amssymb}
\usepackage{graphicx}
\usepackage{multirow}
\usepackage{enumitem,array,algorithm,algpseudocode}
\usepackage{graphicx}
\newsiamthm{claim}{Claim}
\newsiamremark{remark}{Remark}
\newsiamremark{hypothesis}{Hypothesis}
\crefname{hypothesis}{Hypothesis}{Hypotheses}

\newcounter{example}
\newenvironment{example}{\refstepcounter{example}\vspace{1ex}
{\sc Example \theexample.}\hspace{0.3em}\parindent=0pt}{\vspace{1ex}}
\begin{document}

\title{Robust and effective eSIF preconditioning for general SPD
matrices\thanks{Submitted for review.
\funding{The research of Jianlin Xia was supported in part by an NSF grant DMS-1819166.}}
}
\author{Jianlin Xia\thanks{Department of Mathematics, Purdue University, West
Lafayette, IN 47907 (xiaj@math.purdue.\allowbreak~edu).}}
\maketitle

\begin{abstract}
We propose an unconditionally robust and highly effective preconditioner for general symmetric positive definite (SPD)
matrices based on structured
incomplete factorization (SIF), called enhanced SIF (eSIF) preconditioner. The
original SIF strategy proposed recently derives a structured preconditioner by applying block diagonal
preprocessing to the matrix and then compressing
appropriate scaled off-diagonal blocks. Here, we use an enhanced
scaling-and-compression strategy to design the new eSIF preconditioner. Some
subtle modifications are made, such as the use of two-sided block triangular
preprocessing. A practical multilevel eSIF scheme is then designed. We give
rigorous analysis for both the enhanced scaling-and-compression strategy and
the multilevel eSIF preconditioner. The new eSIF framework has some
significant advantages and overcomes some major limitations of the SIF strategy. (i) With the same
tolerance for compressing the off-diagonal blocks,
the eSIF preconditioner can approximate the original matrix to a much higher
accuracy. (ii) The new preconditioner leads to much more significant
reductions of condition numbers due to an accelerated magnification effect for
the decay in the singular values of the scaled off-diagonal blocks. (iii) With
the new preconditioner, the eigenvalues of the preconditioned matrix are much
better clustered around $1$. (iv) The multilevel eSIF preconditioner is
further unconditionally robust or is guaranteed to be positive definite
without the need of extra stabilization, while the multilevel SIF
preconditioner has a strict requirement in order to preserve positive
definiteness. Comprehensive numerical tests are used to show the advantages of
the eSIF preconditioner in accelerating the convergence of iterative solutions.
\end{abstract}

\headers{Jianlin Xia}{Robust and effective eSIF preconditioning}


\section{Introduction\label{sec:intro}}

In this paper, we consider the design of an effective and robust
preconditioning strategy for general dense symmetric positive definite
(SPD)\ matrices. An effective preconditioner can significantly improve the convergence of iterative
solutions. For an SPD matrix $A$, it is also
desirable for the preconditioner to be robust or to preserve the positive
definiteness. A\ commonly used strategy to design robust preconditioners is to
apply modifications or incomplete/approximate Cholesky factorizations to $A$
together with some robustness or stability enhancement strategies (see, e.g.,
\cite{axe94,ben00,ben03,drm94,ker78}).

In recent years, a powerful tool has been introduced into the design of robust
SPD preconditioners and it is to use low-rank approximations for certain dense
blocks in $A$, $A^{-1}$, or some factors of $A$. A common way is to directly
approximate $A$ by so-called rank structured forms like the ones in
\cite{fasthss,xin18}, but it is usually difficult to justify the performance
of the resulting preconditioners. On the other hand, there are two types of
methods that enable rigorous analysis of the effectiveness. One type is in
\cite{li13,li17,li16,xi16} based on low-rank strategies for approximating
$A^{-1}$. Another type is in
\cite{agu18,cam20,fel20,gu10,li12,schurmono,hssprec,hssprecs} where
approximate Cholesky factorizations are computed using low-rank approximations
of relevant off-diagonal blocks. Both types of methods have been shown useful
for many applications. A critical underlying reason (sometimes unnoticed in
earlier work) behind the success of these preconditioners is actually to apply
appropriate block diagonal scaling to $A$ first and then compress the
resulting scaled off-diagonal blocks. A systematic way to formalize this is
given in \cite{hssprec} as a so-called scaling-and-compression strategy and
the resulting factorization is said to be a structured incomplete
factorization (SIF). The preconditioning technique is called SIF\ preconditioning.

The basic idea of (one-level)\ SIF\ preconditioning is as follows \cite{hssprec}. Suppose $A$ is $N\times N$ and is partitioned as
\begin{equation}
A\equiv\left(
\begin{array}
[c]{cc}A_{11} & A_{12}\\
A_{21} & A_{22}\end{array}
\right)  . \label{eq:a}\end{equation}
where the diagonal blocks $A_{11}$ and $A_{22}$ have Cholesky factorizations
of the forms
\begin{equation}
A_{11}=L_{1}L_{1}^{T},\quad A_{22}=L_{2}L_{2}^{T}. \label{eq:diagchol}\end{equation}
Then the inverses of these Cholesky factors are used to scale the off-diagonal
blocks. That is, let
\begin{equation}
C=L_{1}^{-1}A_{12}L_{2}^{-T}. \label{eq:c}\end{equation}
Suppose $C$ has singular values $\sigma_{1}\geq\sigma_{2}\geq\cdots\geq
\sigma_{k}$ (which are actually all smaller than $1$), where $k$ is the
smaller of the row and column sizes of $C$. Then the singular values
$\sigma_{i}$ are truncated aggressively so as to enable the quick computation of a rank structured approximate factorization of $A$.

Thus, the SIF technique essentially employs block diagonal scaling to
preprocess $A$ before relevant compression. This makes a significant difference as compared with standard
rank-structured preconditioners that are based on direct off-diagonal
compression. Accordingly, the SIF preconditioner has some attractive features,
such as the convenient analysis of the performance, the convenient control of
the approximation accuracy, and the nice effectiveness for preconditioning
\cite{hssprec,hssprecs}. In fact, if only $r$ largest singular values of $C$
are kept in its low-rank approximation, then the resulting preconditioner
(called a one-level or prototype preconditioner) approximates $A$ with a
relative accuracy bound $\sigma_{r+1}$. The preconditioner also produces a
condition number $\frac{1+\sigma_{r+1}}{1-\sigma_{r+1}}$ for the
preconditioned matrix. This idea can be repeatedly applied to the diagonal
blocks to yield a practical multilevel SIF preconditioner.

A key idea for the effectiveness of the SIF preconditioner lies in a decay
magnification effect \cite{schurmono,hssprec}. That is, although for a matrix
$A$ where the singular values $\sigma_{i}$ of $C$ may only slightly decay, the
condition number $\frac{1+\sigma_{r+1}}{1-\sigma_{r+1}}$ decays at a much
faster rate to $1$. Thus, it is possible to use a relatively small truncation
rank $r$ to get a structured preconditioner that is both effective and
efficient to apply. A similar reason is also behind the effectiveness of those
preconditioners in \cite{li13,li17,li16,xi16,schurmono}.

However, the SIF\ preconditioning has two major limitations. One is in the
robustness. In the multilevel case, it needs a strict condition to avoid
breakdown and ensure the existence or positive definiteness of the
preconditioner. This condition needs either the condition number of $A$ to be
reasonably small, the low-rank approximation tolerance to be small, or the
number of levels to be small. These mean the sacrifice of either the applicability or the efficiency of the preconditioner, as pointed out in
\cite{hssprecs}.

Another limitation is in the effectiveness. Although the condition number form
$\frac{1+\sigma_{r+1}}{1-\sigma_{r+1}}$ has the decay magnification effect, if
the decay of $\sigma_{i}$ is too slow, using small $r$ would not reduce the
condition number too much. With small $r$, the eigenvalues of the
preconditioned matrix may not closely cluster around $1$ either. The
performance of the preconditioner can then be less satisfactory.

Therefore, the motivation of this work is to overcome both limitations of the
SIF\ technique. We make enhancements in several aspects. First, we would like
get rid of the condition in the SIF scheme that avoids breakdown. That is, we
produce a type of structured preconditioners that is \emph{unconditionally
robust} or always positive definite. Second, we would like to approximate $A$
with \emph{better accuracies} using the same truncation rank $r$. Next, we
intend to \emph{accelerate the decay magnification effect} in the condition
number form. Lastly, we also try to \emph{improve the eigenvalue clustering}
of the preconditioned matrix.

Our idea to achieve these enhancements is to make some subtle changes to the
original SIF scheme. Instead of block diagonal scaling, we use two-sided
block triangular preprocessing which leads to an \emph{enhanced
scaling-and-compression strategy}. Then a low-rank
approximation is still computed for $C$, but it is just used to accelerate
computations related to Schur complements instead of off-diagonal blocks.
(This will be made more precise in Section \ref{sec:prototype}.) This strategy
can be repeatedly applied to $A_{11}$ and $A_{22}$ in (\ref{eq:a}) so as to
yield an efficient structured multilevel preconditioner.

This strategy makes it convenient to analyze the resulting preconditioners.
The one-level preconditioner can now approximate $A\ $with a relative accuracy
bound $\sigma_{r+1}^{2}$ (in contrast with the bound $\sigma_{r+1}$ in the
SIF\ case). The preconditioned matrix now has condition number $\frac
{1}{1-\sigma_{r+1}^{2}}$, which is a significant improvement from
$\frac{1+\sigma_{r+1}}{1-\sigma_{r+1}}$ due to the quadratic form
$\sigma_{r+1}^{2}$ and the smaller numerator. Similar improvements are also
achieved with the multilevel preconditioner.

Moreover, the eigenvalues of the preconditioned matrix are now more closely
clustered around $1$. With the new one-level preconditioner, the eigenvalues
are redistributed to $[1-\sigma_{r+1}^{2},1]$, with the eigenvalue $1$ of
multiplicity $N-(k-r)$. In comparison, the one-level SIF preconditioner only
brings the eigenvalues to the interval $[1-\sigma_{r+1},1+\sigma_{r+1}]$, with
the eigenvalue $1$ of multiplicity $N-2(k-r)$. Similarly, the new multilevel
preconditioner also greatly improves the eigenvalue clustering.

In addition, the multilevel generalization of the strategy always produces a
positive definite preconditioner $\tilde{A}$ without the need of extra
stabilization or diagonal compensation. In fact, the scheme has an automatic
\emph{positive definiteness enhancement effect}. That is, $\tilde{A}$ is equal
to $A$ plus a positive semidefinite matrix. Thus, the new multilevel
preconditioner is unconditionally robust.

Due to all these enhancements, the new preconditioner is called an
\emph{enhanced SIF (eSIF) preconditioner}. We give comprehensive analysis of
the accuracy, robustness, and effectiveness of both the one-level and the
multilevel eSIF preconditioners in Theorems \ref{thm:err1}, \ref{thm:ev1},
\ref{thm:err} and \ref{thm:ev}. All the benefits combined yield significantly
better effectiveness than the SIF scheme. With the same number of levels and
the same truncation rank $r$, although the eSIF preconditioner is slightly
more expensive to apply in each iteration step, the total iterative solution
cost is much lower.

We also show some techniques to design a practical multilevel eSIF scheme and
then analyze the efficiency and storage. The practical scheme avoids forming
dense blocks like $C$ in (\ref{eq:c}) while enabling the convenient low-rank
approximation of these blocks. It also produces structured factors defined by
compact forms such as Householder vectors.

The performance of the preconditioner is illustrated in terms of some
challenging test matrices including some from \cite{hssprec}. As compared with
the SIF preconditioner, the eSIF preconditioner yields dramatic reductions in
the number of conjugate gradient iterations.

The organization of the remaining sections is as follows. The enhanced
scaling-and-compression strategy and the one-level eSIF preconditioner will be
presented and analyzed in Section \ref{sec:prototype}. The techniques and
analysis will then be generalized to multiple levels in Section
\ref{sec:multi}. Section \ref{sec:alg} further gives the practical multilevel
design of the preconditioning scheme and also analyzes the storage and costs.
Comprehensive numerical tests will be given in Section \ref{sec:tests},
following by some conclusions and discussions in Section \ref{sec:concl}. For
convenience, we list frequently used notation as follows.

\begin{itemize}
\item $\lambda(A)$ is used to represent an eigenvalue of $A$ (it is used in a
general way and is not for any specific eigenvalue).

\item $\kappa(A)$ denotes the 2-norm condition number of $A$.

\item $\operatorname{diag}(\cdot)$ is used to mean a diagonal or block
diagonal matrix constructed with the given diagonal entries or blocks.

\item $I_{n}$ is the $n\times n$ identity matrix and is used to distinguish
identity matrices of different sizes in some contexts.
\end{itemize}

\section{Enhanced scaling-and-compression strategy and prototype eSIF
preconditioner\label{sec:prototype}}

We first give the enhanced scaling-and-compression strategy and analyze the
resulting prototype eSIF preconditioner in terms of the accuracy, robustness, and effectiveness.

In the SIF preconditioner in \cite{hssprec}, $A$ in (\ref{eq:a}) can be
written as a factorized form as follows based on (\ref{eq:diagchol}) and
(\ref{eq:c}):
\begin{equation}
A=\left(
\begin{array}
[c]{cc}L_{1} & \\
& L_{2}\end{array}
\right)  \left(
\begin{array}
[c]{cc}I & C\\
C^{T} & I
\end{array}
\right)  \left(
\begin{array}
[c]{cc}L_{1}^{T} & \\
& L_{2}^{T}\end{array}
\right)  , \label{eq:scale0}\end{equation}
where $\left(
\begin{array}
[c]{cc}I & C\\
C^{T} & I
\end{array}
\right)  $
can be viewed as the result after the block diagonal preprocessing or scaling of $A$. $C$ is then approximated by a low-rank form so as to obtain a rank-structured approximate factorization of $A$.

Here, we make some subtle changes which will turn out to make a significant
difference. Rewrite (\ref{eq:scale0}) in the following form:\begin{equation}
A=\left(
\begin{array}
[c]{cc}L_{1} & \\
L_{2}C^{T} & L_{2}\end{array}
\right)  \left(
\begin{array}
[c]{cc}I & \\
& I-C^{T}C
\end{array}
\right)  \left(
\begin{array}
[c]{cc}L_{1}^{T} & CL_{2}^{T}\\
& L_{2}^{T}\end{array}
\right)  . \label{eq:ascale}\end{equation}
Suppose $C$ is $m\times n$ and a rank-$r$ truncated SVD of $C$ is
\begin{equation}
C\approx U_{1}\Sigma_{1}V_{1}^{T}, \label{eq:svd}\end{equation}
where $\Sigma_{1}=\operatorname{diag}(\sigma_{1},\sigma_{2},\ldots,\sigma
_{r})$ is for the largest $r$ singular values $\sigma_{1}\geq\sigma_{2}\geq\cdots\geq\sigma_{r}$ of $C$. For later convenience, we also let the full
SVD of $C$ be\begin{equation}
C=U\Sigma V^{T}=U_{1}\Sigma_{1}V_{1}^{T}+U_{2}\Sigma_{2}V_{2}^{T},
\label{eq:svdc}\end{equation}
where $U=\left(
\begin{array}
[c]{cc}U_{1} & U_{2}\end{array}
\right)  $ and $V=\left(
\begin{array}
[c]{cc}V_{1} & V_{2}\end{array}
\right)  $ are orthogonal and $\Sigma_{2}$ is a (rectangular) diagonal matrix
for the remaining singular values $\sigma_{r+1}\geq\cdots\geq\sigma
_{\min\{m,n\}}$. We further suppose $\tau$ is a tolerance for truncating the
singular values in (\ref{eq:svd}). That is,
\begin{equation}
\sigma_{r}\geq\tau\geq\sigma_{r+1}. \label{eq:tau}\end{equation}
Note that all the singular values $\sigma_{i}$ of $C$ satisfy $\sigma_{i}<1$
\cite{hssprec}, so $\tau<1$.

The apply (\ref{eq:svd}) to $C^{T}C$ in (\ref{eq:ascale}) to get\[
C^{T}C\approx V_{1}\Sigma_{1}^{2}V_{1}^{T}.
\]
In the meantime, we preserve the original form of $C$ in the two triangular factors in (\ref{eq:ascale}). Accordingly,
\begin{equation}
A\approx\tilde{A}\equiv\left(
\begin{array}
[c]{cc}L_{1} & \\
L_{2}C^{T} & L_{2}\end{array}
\right)  \left(
\begin{array}
[c]{cc}I & \\
& I-V_{1}\Sigma_{1}^{2}V_{1}^{T}\end{array}
\right)  \left(
\begin{array}
[c]{cc}L_{1}^{T} & CL_{2}^{T}\\
& L_{2}^{T}\end{array}
\right)  \label{eq:tildea0}\end{equation}
Suppose $\tilde{D}_{2}$ is the lower triangular Cholesky factor of
$I-V_{1}\Sigma_{1}^{2}V_{1}^{T}$:
\begin{equation}
I-V_{1}\Sigma_{1}^{2}V_{1}^{T}=\tilde{D}_{2}\tilde{D}_{2}^{T}.
\label{eq:d2tilde}\end{equation}
Let\begin{equation}
\tilde{L}=\left(
\begin{array}
[c]{cc}L_{1} & \\
L_{2}C^{T} & L_{2}\end{array}
\right)  \left(
\begin{array}
[c]{cc}I & \\
& \tilde{D}_{2}\end{array}
\right)  =\left(
\begin{array}
[c]{cc}L_{1} & \\
& L_{2}\end{array}
\right)  \left(
\begin{array}
[c]{cc}I & \\
C^{T} & I
\end{array}
\right)  \left(
\begin{array}
[c]{cc}I & \\
& \tilde{D}_{2}\end{array}
\right)  . \label{eq:l0}\end{equation}
Then we get a \emph{prototype (1-level) eSIF preconditioner}
\begin{equation}
\tilde{A}=\tilde{L}\tilde{L}^{T}. \label{eq:prec0}
\end{equation}

This scheme can be understood as follows. Unlike in the SIF scheme where $A$
is preprocessed by the block diagonal factor $\left(
\begin{array}
[c]{cc}L_{1} & \\
& L_{2}\end{array}
\right)  $, here we use a block triangular factor $\left(
\begin{array}
[c]{cc}L_{1} & \\
L_{2}C^{T} & L_{2}\end{array}
\right)  $ to preprocess $A$. Note that it is still convenient to invert
$\left(
\begin{array}
[c]{cc}L_{1} & \\
L_{2}C^{T} & L_{2}\end{array}
\right)  =\left(
\begin{array}
[c]{cc}L_{1} & \\
& L_{2}\end{array}
\right)  \left(
\begin{array}
[c]{cc}I & \\
C^{T} & I
\end{array}
\right)  $ in linear system solution so the form of $C$ does not cause any
substantial trouble. Also, we do not need to explicitly form or compress $C$.
In addition, the Cholesky factor $\tilde{D}_{2}$ in (\ref{eq:d2tilde}) is only
used for the purpose of analysis and does not need to be computed. The details
will be given later in a more practical scheme in Section \ref{sec:alg}.

This leads to our \emph{enhanced scaling-and-compression strategy}. We then
analyze the properties of the resulting prototype eSIF preconditioner.
Obviously, $\tilde{A}$ in (\ref{eq:prec0}) always exists and is positive
definite. Furthermore, an additional benefit in the positive definiteness can
be shown. We take a closer look at the positive definiteness of $\tilde{A}$
and also the accuracy of $\tilde{A}$ for approximating $A$.

\begin{theorem}
\label{thm:err1}Let $\tau$ be the truncation tolerance in (\ref{eq:tau}).
$\tilde{A}$ in (\ref{eq:prec0}) satisfies\[
\tilde{A}=A+E,
\]
where $E$ is a positive semidefinite matrix and
\begin{equation}
\frac{\Vert E\Vert_{2}}{\Vert A\Vert_{2}}\leq\sigma_{r+1}^{2}\leq\tau^{2}.
\label{eq:e}\end{equation}
In addition,
\begin{equation}
\frac{\Vert\tilde{L}-L\Vert_{2}}{\Vert L\Vert_{2}}\leq\frac{c\sqrt
{1-\sigma_{n}^{2}}}{1-\sigma_{1}^{2}}\tau^{2}, \label{eq:cholerr}\end{equation}
where $L$ is the lower triangular Cholesky factor of $A$, $c=1+2\left\lceil
\log_{2}n\right\rceil $, and $\sigma_{n}$ is either the $n$-th singular value
of $C$ when $m\geq n$ or is $0$ otherwise. On the other hand, if $\tilde
{D}_{2}$ in $\tilde{L}$ in (\ref{eq:l0}) is replaced by $(I-V_{1}\Sigma
_{1}^{2}V_{1}^{T})^{1/2}$ and $L$ is modified accordingly as $L=\left(
\begin{array}
[c]{cc}L_{1} & \\
L_{2}C^{T} & L_{2}(I-V\Sigma^{T}\Sigma V^{T})^{1/2}\end{array}
\right)  $ so that $A=LL^{T}$ still holds, then
\begin{equation}
\frac{\Vert\tilde{L}-L\Vert_{2}}{\Vert L\Vert_{2}}<\tau^{2}.
\label{eq:cholerr2}\end{equation}

\end{theorem}

\begin{proof}
From (\ref{eq:svdc}) and (\ref{eq:tildea0}), $\tilde{A}$ can be written as\begin{align*}
\tilde{A}  &  =\left(
\begin{array}
[c]{cc}A_{11} & A_{12}\\
A_{21} & L_{2}C^{T}CL_{2}^{T}+L_{2}(I-V_{1}\Sigma_{1}^{2}V_{1}^{T})L_{2}^{T}\end{array}
\right) \\
&  =\left(
\begin{array}
[c]{cc}A_{11} & A_{12}\\
A_{21} & A_{22}+L_{2}(C^{T}C-V_{1}\Sigma_{1}^{2}V_{1}^{T})L_{2}^{T}\end{array}
\right) \\
&  =\left(
\begin{array}
[c]{cc}A_{11} & A_{12}\\
A_{21} & A_{22}+L_{2}(V_{2}\Sigma_{2}^{T}\Sigma_{2}V_{2}^{T})L_{2}^{T}\end{array}
\right)  =A+E,
\end{align*}
where $E=\operatorname*{diag}(0,L_{2}(V_{2}\Sigma_{2}^{T}\Sigma_{2}V_{2}^{T})L_{2}^{T})$ is positive semidefinite and\[
\Vert E\Vert_{2}=\Vert L_{2}(V_{2}\Sigma_{2}^{T}\Sigma_{2}V_{2}^{T})L_{2}^{T}\Vert_{2}\leq\sigma_{r+1}^{2}\Vert L_{2}\Vert_{2}^{2}=\sigma_{r+1}^{2}\Vert A_{22}\Vert_{2}\leq\sigma_{r+1}^{2}\Vert A\Vert_{2}.
\]
Also, let $D_{2}D_{2}^{T}=I-V\Sigma^{T}\Sigma V^{T}$. Then $L=\left(
\begin{array}
[c]{cc}L_{1} & \\
L_{2}C^{T} & L_{2}D_{2}\end{array}
\right)  $. Thus,\begin{align}
\Vert\tilde{L}-L\Vert_{2}  &  =\left\Vert \left(
\begin{array}
[c]{cc}L_{1} & \\
L_{2}C^{T} & L_{2}\tilde{D}_{2}\end{array}
\right)  -\left(
\begin{array}
[c]{cc}L_{1} & \\
L_{2}C^{T} & L_{2}D_{2}\end{array}
\right)  \right\Vert _{2}\label{eq:lerror}\\
&  =\left\Vert \left(
\begin{array}
[c]{cc}0 & \\
& L_{2}(\tilde{D}_{2}-D_{2})
\end{array}
\right)  \right\Vert _{2}\leq\Vert L\Vert_{2}\Vert\tilde{D}_{2}-D_{2}\Vert
_{2}.\nonumber
\end{align}
When $D_{2}$ is the lower triangular Cholesky factor of $I-V\Sigma^{T}\Sigma
V^{T}$, an inequality in \cite{hssprec} gives
\[
\Vert\tilde{D}_{2}-D_{2}\Vert_{2}\leq\frac{c\sqrt{1-\sigma_{n}^{2}}}{1-\sigma_{1}^{2}}\sigma_{r+1}^{2},\quad c=1+2\left\lceil \log_{2}n\right\rceil .
\]
This leads to (\ref{eq:cholerr}).

If $\tilde{D}_{2}$ in $\tilde{L}$ is replaced by $(I-V_{1}\Sigma_{1}^{2}V_{1}^{T})^{1/2}$ and $D_{2}$ is replaced by $(I-V\Sigma^{T}\Sigma
V^{T})^{1/2}$, then\begin{align*}
\Vert\tilde{D}_{2}-D_{2}\Vert_{2}  &  =\Vert(I-V_{1}\Sigma_{1}^{2}V_{1}^{T})^{1/2}-(I-V\Sigma^{T}\Sigma V^{T})^{1/2}\Vert_{2}\\
&  =\Vert(I-\operatorname{diag}(\Sigma_{1}^{2},0))^{1/2}-(I-\Sigma^{T}\Sigma)^{1/2}\Vert_{2}\\
&  =1-\sqrt{1-\sigma_{r+1}^{2}}<\sigma_{r+1}^{2}.
\end{align*}
Then following (\ref{eq:lerror}), we get (\ref{eq:cholerr2}).
\end{proof}

This theorem gives both the accuracy and the robustness of the prototype eSIF
preconditioner. Unlike the SIF\ framework where a similar prototype
preconditioner has a relative accuracy bound $\tau$, here the bound is
$\tau^{2}$ that is much more accurate. In addition, this theorem means the
construction of $\tilde{A}$ automatically has a \emph{positive definiteness
enhancement effect}: it implicitly compensates $A$ by a positive semidefinite
matrix $E$. This is similar to ideas in \cite{gu10,schurmono}. Later, we will
show that this effect further carries over to the multilevel generalization,
which is not the case for the SIF\ preconditioner.

The effectiveness of the prototype eSIF preconditioner can be shown as follows.

\begin{theorem}
\label{thm:ev1}The eigenvalues of $\tilde{L}^{-1}A\tilde{L}^{-T}$ are\[
\lambda(\tilde{L}^{-1}A\tilde{L}^{-T})=1-\sigma_{r+1}^{2},\ldots,1-\sigma
_{k}^{2},\underset{N-(k-r)}{\underbrace{1,\ldots,1}},
\]
where $k=\min\{m,n\}$. Accordingly,
\begin{align*}
\Vert\tilde{L}^{-1}A\tilde{L}^{-T}-I\Vert_{2}  &  =\sigma_{r+1}^{2}\leq
\tau^{2},\\
\kappa(\tilde{L}^{-1}A\tilde{L}^{-T})  &  =\frac{1}{1-\sigma_{r+1}^{2}}\leq\frac{1}{1-\tau^{2}}.
\end{align*}

\end{theorem}

\begin{proof}
It is not hard to verify\begin{equation}
\tilde{L}^{-1}A\tilde{L}^{-T}=\operatorname*{diag}(I_{N-n},\tilde{D}_{2}^{-1}(I_{n}-V\Sigma^{T}\Sigma V^{T})\tilde{D}_{2}^{-T}). \label{eq:aprec0}\end{equation}
The eigenvalues of $\tilde{D}_{2}^{-1}(I_{n}-V\Sigma^{T}\Sigma V^{T})\tilde
{D}_{2}^{-T}$ are
\begin{align*}
\lambda(\tilde{D}_{2}^{-1}(I_{n}-V\Sigma^{T}\Sigma V^{T})\tilde{D}_{2}^{-T})
&  =\lambda(\tilde{D}_{2}^{-T}\tilde{D}_{2}^{-1}(I_{n}-V\Sigma^{T}\Sigma
V^{T}))\\
&  =\lambda((I_{n}-V_{1}\Sigma_{1}^{2}V_{1}^{T})^{-1}(I_{n}-V\Sigma^{T}\Sigma
V^{T})).
\end{align*}
Further derivations can be done via the Sherman-Morrison-Woodbury formula or
in the following way:\begin{align*}
&  (I_{n}-V_{1}\Sigma_{1}^{2}V_{1}^{T})^{-1}(I_{n}-V\Sigma^{T}\Sigma V^{T})\\
=\  &  (V(I_{n}-\operatorname{diag}(\Sigma_{1}^{2},0))V^{T})^{-1}V(I_{n}-\Sigma^{T}\Sigma)V^{T}\\
=\  &  V\operatorname{diag}((I_{r}-\Sigma_{1}^{2})^{-1},I_{n-r})(I_{n}-\Sigma^{T}\Sigma)V^{T}\\
=\  &  V\operatorname{diag}(I_{r},I_{n-r}-\Sigma_{2}^{T}\Sigma_{2})V^{T}.
\end{align*}
Thus,
\begin{equation}
\lambda(\tilde{D}_{2}^{-1}(I_{n}-V\Sigma^{T}\Sigma V^{T})\tilde{D}_{2}^{-T})=\lambda(\operatorname*{diag}(I_{r},I_{n-r}-\Sigma_{2}^{T}\Sigma_{2})),
\label{eq:eig3}\end{equation}
which are just $1-\sigma_{r+1}^{2},\ldots,1-\sigma_{k}^{2},1$. The eigenvalue
$1$ is a multiple eigenvalue. If $k=n$, then the eigenvalue $1$ in
(\ref{eq:eig3}) has multiplicity $r$. If $k=m$, $I_{n-r}-\Sigma_{2}^{T}\Sigma_{2}$ also has $n-k$ eigenvalues equal to $1$ so the eigenvalue $1$ in
(\ref{eq:eig3}) has multiplicity $n-(k-r)$. For both cases, the eigenvalue $1$
of $\tilde{L}^{-1}A\tilde{L}^{-T}$ has multiplicity $N-(k-r)$ according to
(\ref{eq:aprec0}).
\end{proof}

To give an idea on the advantages of the prototype eSIF preconditioner over
the corresponding prototype SIF preconditioner in \cite{hssprec}, we compare
the results in Table \ref{tab:compare1} with $\tilde{L}$ and $\tilde{A}$ from
the eSIF or SIF scheme. The eSIF scheme yields a much higher approximation
accuracy than SIF ($\tau^{2}$ vs. $\tau$) for both $\frac{\Vert A-\tilde
{A}\Vert_{2}}{\Vert A\Vert_{2}}$ and $\Vert\tilde{L}^{-1}A\tilde{L}^{-T}-I\Vert_{2}$. The eigenvalues of the preconditioned matrix $\tilde
{L}^{-1}A\tilde{L}^{-T}$ from eSIF are also much more closely clustered around
$1$ and eSIF produces a lot more eigenvalues equal to $1$ than SIF. This is further illustrated in Figure \ref{fig:ev1}.

\begin{table}[ptbh]
\caption{Comparison of prototype SIF and eSIF\ preconditioners that are used
to produce $\tilde{L}$ and $\tilde{A}$, where $k=\min\{m,n\}$ and the results
for the SIF preconditioner are from \cite{hssprec,hssprecs}.}\label{tab:compare1}
\begin{center}
\tabcolsep2pt\renewcommand{\arraystretch}{1.6}
\begin{tabular}
[c]{c|c|c}\hline
& SIF & eSIF\\\hline
$\frac{\Vert\tilde{A}-A\Vert_{2}}{\Vert A\Vert_{2}}$ & $\leq\tau$ & $\leq
\tau^{2}$\\
$\frac{\Vert\tilde{L}-L\Vert_{2}}{\Vert L\Vert_{2}}$ & $\leq\tau+\frac
{c\sqrt{1-\sigma_{n}^{2}}}{1-\sigma_{1}^{2}}\tau^{2}$ & $\leq\frac
{c\sqrt{1-\sigma_{n}^{2}}}{1-\sigma_{1}^{2}}\tau^{2}$\\
$\lambda(\tilde{L}^{-1}A\tilde{L}^{-T})$ & $1\pm\sigma_{r+1},\ldots,1\pm
\sigma_{k},\underset{N-2(k-r)}{\underbrace{1,\ldots,1}}$ & $1-\sigma_{r+1}^{2},\ldots,1-\sigma_{k}^{2},\underset{N-(k-r)}{\underbrace{1,\ldots,1}}$\\
$\Vert\tilde{L}^{-1}A\tilde{L}^{-T}-I\Vert_{2}$ & $\sigma_{r+1}\leq\tau$ &
$\sigma_{r+1}^{2}\leq\tau^{2}$\\
$\kappa(\tilde{L}^{-1}A\tilde{L}^{-T})$ & $\frac{1+\sigma_{r+1}}{1-\sigma_{r+1}}\leq\frac{1+\tau}{1-\tau}$ & $\frac{1}{1-\sigma_{r+1}^{2}}\leq\frac{1}{1-\tau^{2}}$\\\hline
\end{tabular}
\end{center}
\end{table}

\begin{figure}[ptbh]
\centering\includegraphics[height=0.8in]{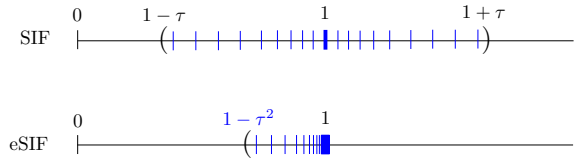}\caption{How the
eigenvalues $\lambda(\tilde{L}^{-1}A\tilde{L}^{-T})$ cluster around $1$ when
$\tilde{L}\tilde{L}^{T}$ is obtained with the prototype SIF and
eSIF\ preconditioners.}\label{fig:ev1}\end{figure}

Specifically, SIF produces $\kappa(\tilde{L}^{-1}A\tilde{L}^{-T})=\frac{1+\sigma_{r+1}}{1-\sigma_{r+1}}$, while eSIF\ leads to much smaller
$\kappa(\tilde{L}^{-1}A\tilde{L}^{-T})=\frac{1}{1-\sigma_{r+1}^{2}}$. (Notice
the quadratic term $\sigma_{r+1}^{2}$ in the denominator and the smaller
numerator.) To further illustrate the difference in $\kappa(\tilde{L}^{-1}A\tilde{L}^{-T})$, we use an example like in \cite{hssprec}. In the
example, the singular values of $C$ look like those in Figure \ref{fig:lap2d}(a) and are based on the analytical forms from a 5-point discrete Laplacian
matrix \cite{hssprecs}. The singular values of $C$ in (\ref{eq:c}) only slowly
decay. Figure \ref{fig:lap2d}(b) shows $\kappa(\tilde{L}^{-1}A\tilde{L}^{-T})$
from both schemes. We can observe two things.

\begin{enumerate}
\item Like in SIF, the modest decay of the nonzero singular values $\sigma
_{i}$ of $C$ is further dramatically magnified in $\frac{1}{1-\sigma_{i}^{2}}$. That is, even if $\sigma_{i}$ decays slowly, $\frac{1}{1-\sigma_{i}^{2}}$
decays much faster so that $\sigma_{i}$ can still be aggressively truncated so
as to produce reasonably small $\kappa(\tilde{L}^{-1}A\tilde{L}^{-T})$. This
is the \emph{decay magnifying effect} like in \cite{hssprec}.

\item Furthermore, the decay magnification effect from eSIF is more dramatic
since$\ \frac{1}{1-\sigma_{i}^{2}}$ is smaller than $\frac{1+\sigma_{i}}{1-\sigma_{i}}$ by a factor of $(1+\sigma_{i})^{2}$. For a large range of $r$
values, eSIF gives much better condition numbers than SIF.
\end{enumerate}

\begin{figure}[ptbh]
\centering\begin{tabular}
[c]{cc}\includegraphics[height=1.6in]{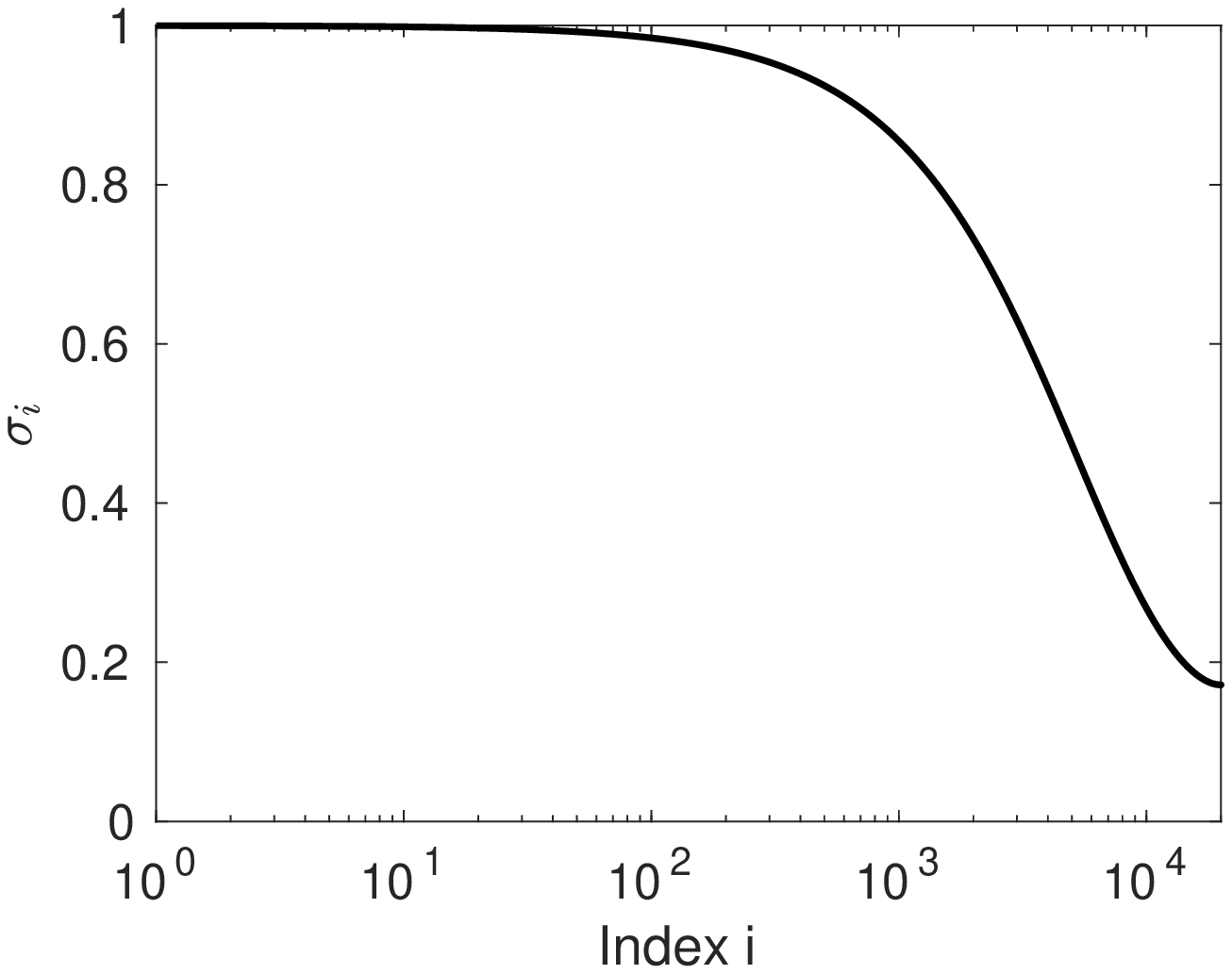} &
\includegraphics[height=1.6in]{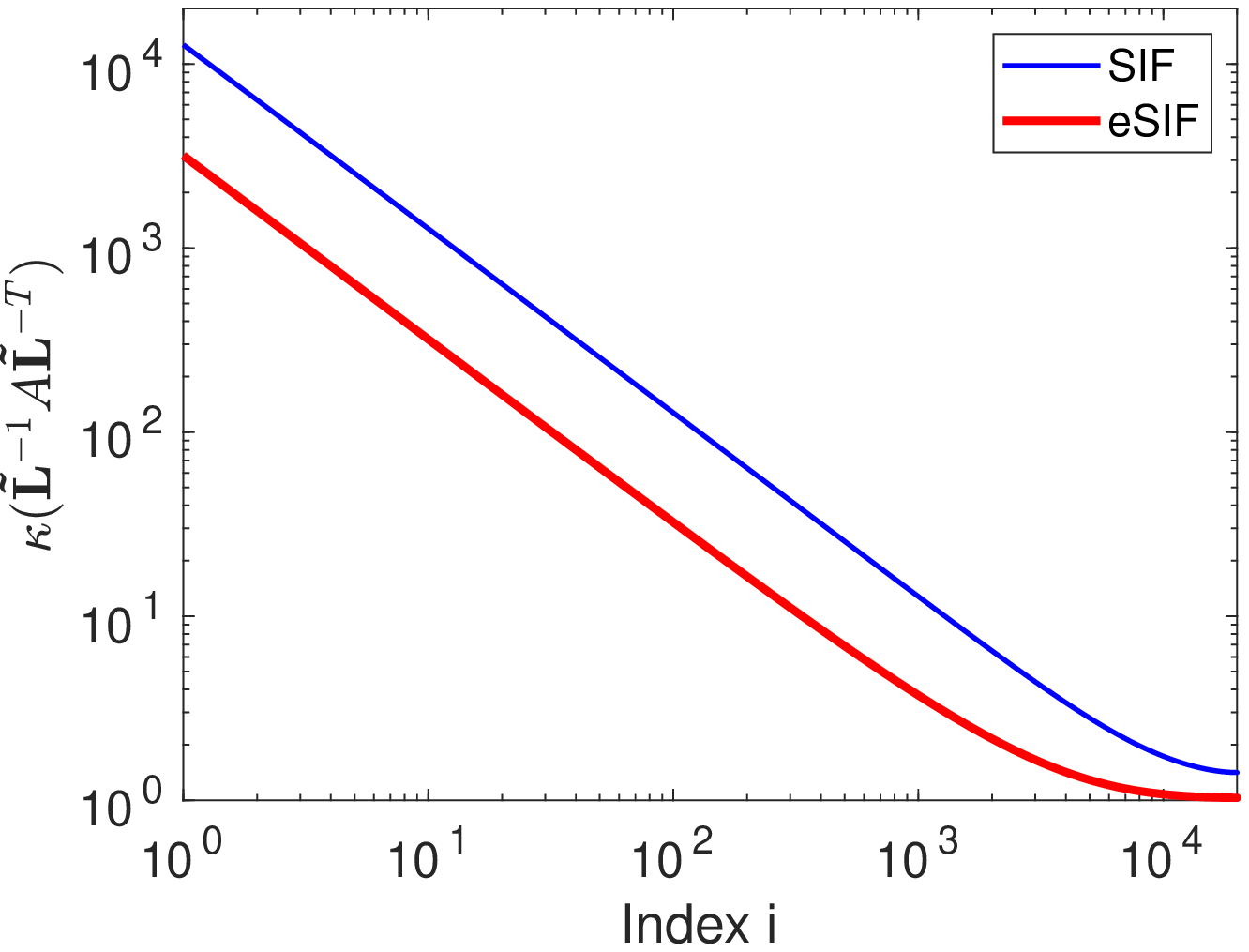}\\
{\scriptsize (a)\ Singular values $\sigma_{i}$ of $C$} &
{\scriptsize (b)\ $\kappa(\tilde{L}^{-1}A\tilde{L}^{-T})$}\end{tabular}
\caption{For an example where the singular values $\sigma_{i}$ of $C$ slowly
decay, how $\kappa(\tilde{L}^{-1}A\tilde{L}^{-T})$ decays when $\tilde{L}$ is
from the prototype SIF or eSIF preconditioner obtained by truncating
$\sigma_{i}$ with $r$ set to be $i$ in (b).}\label{fig:lap2d}\end{figure}

\section{Multilevel eSIF preconditioner\label{sec:multi}}

The prototype preconditioner in the previous section still has two dense
Cholesky factors $L_{1}$ and $L_{2}$ in (\ref{eq:l0}). To get an efficient
preconditioner, we generalize the prototype preconditioner to multiple levels.
That is, apply it repeatedly to the diagonal blocks of $A$. For convenience,
we use eSIF(1) to denote the prototype $1$-level eSIF scheme. A 2-level eSIF
scheme or eSIF(2) uses eSIF(1) to obtain approximate factors $\tilde{L}_{1}\approx L_{1}$ and $\tilde{L}_{2}\approx L_{2}$ for (\ref{eq:diagchol}).
Similarly, an $l$-level eSIF scheme or eSIF($l$) uses eSIF($l-1$) to
approximate $L_{1}$ and $L_{2}$. With a sufficient number of levels (usually
$l=O(\log N)$), the finest level diagonal blocks are small enough and can be
directly factorized. The overall resulting factor $\tilde{L}$ is an eSIF($l$)
factor. The resulting approximation matrix $\tilde{A}$ is an eSIF($l$) preconditioner.

We prove that the eSIF($l$) preconditioner $\tilde{A}$ is always positive
definite and show how accurate $\tilde{A}$ is for approximating $A$.

\begin{theorem}
\label{thm:err}Let $\tau$ be the tolerance for any singular value truncation
like (\ref{eq:svd})--(\ref{eq:tau}) in the eSIF($l$) scheme. The approximate
matrix $\tilde{A}$ resulting from eSIF($l$) is always positive definite and
satisfies\begin{equation}
\tilde{A}=A+E, \label{eq:tildea}\end{equation}
where $E$ is a positive semidefinite matrix and
\[
\frac{\Vert E\Vert_{2}}{\Vert A\Vert_{2}}\leq(1+\tau^{2})^{l}-1.
\]

\end{theorem}

\begin{proof}
We prove this by induction. $l=1$ corresponds to eSIF(1) and the result is in
Theorem \ref{thm:err1}. Suppose the result holds for eSIF($l-1$) with $l>1$.
Apply eSIF($l-1$) to $A_{11}$ and $A_{22}$ to get approximate Cholesky factors $\tilde{L}_{1}$ and $\tilde{L}_{2}$, respectively. By induction, we have
\[
\tilde{L}_{1}\tilde{L}_{1}{}^{T}=A_{11}+E_{1},\quad\tilde{L}_{2}\tilde{L}_{2}{}^{T}=A_{22}+E_{2},
\]
where $E_{1}$ and $E_{2}$ are positive semidefinite matrices satisfying\begin{align*}
\Vert E_{1}\Vert_{2}  &  \leq\left[  (1+\tau^{2})^{l-1}-1\right]  \Vert
A_{11}\Vert_{2}\leq\left[  (1+\tau^{2})^{l-1}-1\right]  \Vert A\Vert_{2},\\
\Vert E_{2}\Vert_{2}  &  \leq\left[  (1+\tau^{2})^{l-1}-1\right]  \Vert
A_{22}\Vert_{2}\leq\left[  (1+\tau^{2})^{l-1}-1\right]  \Vert A\Vert_{2}.
\end{align*}
Thus,\[
A\approx\left(
\begin{array}
[c]{cc}\tilde{L}_{1}\tilde{L}_{1}{}^{T} & A_{21}^{T}\\
A_{21} & \tilde{L}_{2}\tilde{L}_{2}{}^{T}\end{array}
\right)  =A+\operatorname*{diag}(E_{1},E_{2})\equiv\hat{A}.
\]
Clearly, $\hat{A}$ is always positive definite.

Then apply eSIF(1) to $\hat{A}$ to yield\[
\hat{A}\approx\tilde{A}\equiv\tilde{L}\tilde{L}^{T},
\]
where $\tilde{L}$ is the eSIF($l$) factor. With Theorem \ref{thm:err1} applied
to $\hat{A}$, we get\[
\tilde{A}=\hat{A}+\tilde{E},
\]
where $\tilde{E}$ is a positive semidefinite matrix satisfying $\Vert\tilde
{E}\Vert_{2}\leq\tau^{2}\Vert\hat{A}\Vert_{2}$. Then\[
\tilde{A}=A+(\operatorname*{diag}(E_{1},E_{2})+\tilde{E})\equiv A+E,
\]
where $E=\operatorname*{diag}(E_{1},E_{2})+\tilde{E}$ is positive
semidefinite. Thus, $\tilde{A}$ is positive definite and
\begin{align*}
\Vert E\Vert_{2}  &  \leq\Vert\operatorname*{diag}(E_{1},E_{2})\Vert_{2}+\Vert\tilde{E}\Vert_{2}\\
&  \leq\Vert\operatorname*{diag}(E_{1},E_{2})\Vert_{2}+\tau^{2}\Vert\hat
{A}\Vert_{2}\\
&  =\Vert\operatorname*{diag}(E_{1},E_{2})\Vert_{2}+\tau^{2}\Vert
A+\operatorname*{diag}(E_{1},E_{2})\Vert_{2}\\
&  \leq\tau^{2}\Vert A\Vert_{2}+(1+\tau^{2})\Vert\operatorname*{diag}(E_{1},E_{2})\Vert_{2}\\
&  \leq\tau^{2}\Vert A\Vert_{2}+(1+\tau^{2})\left[  (1+\tau^{2})^{l-1}-1\right]  \Vert A\Vert_{2}\\
&  =\left[  (1+\tau^{2})^{l}-1\right]  \Vert A\Vert_{2}.
\end{align*}
The result then holds by induction.
\end{proof}

Thus, $\frac{\Vert E\Vert_{2}}{\Vert A\Vert_{2}}$ is roughly $O(l\tau^{2})$
for reasonable $\tau$, which indicates a very slow levelwise approximation
error accumulation. Moreover, like eSIF($1$), eSIF($l$) also has a positive
definiteness enhancement effect so that $\tilde{A}$ remains positive definite.
In contract, the multilevel SIF\ scheme in \cite{hssprec} may breakdown due to the loss of positive definiteness.

Then we can look at the effectiveness of the eSIF($l$) preconditioner.

\begin{theorem}
\label{thm:ev}Let $\tau$ be the tolerance for any singular value truncation
like (\ref{eq:svd})--(\ref{eq:tau}) in the eSIF($l$) scheme and $\epsilon
=\left[  (1+\tau^{2})^{l}-1\right]  \kappa(A)$. Let $\tilde{L}$ be the
eSIF($l$) factor. Then the eigenvalues of the preconditioned matrix $\tilde
{L}^{-1}A\tilde{L}^{-T}$ satisfy\begin{equation}
\frac{1}{1+\epsilon}\leq\lambda(\tilde{L}^{-1}A\tilde{L}^{-T})\leq1.
\label{eq:eig2}\end{equation}
Accordingly,\begin{align*}
\Vert\tilde{L}^{-1}A\tilde{L}^{-T}-I\Vert_{2}  &  \leq\frac{\epsilon
}{1+\epsilon},\\
\kappa(\tilde{L}^{-1}A\tilde{L}^{-T})  &  \leq1+\epsilon.
\end{align*}

\end{theorem}

\begin{proof}
Let $A=LL^{T}$ be the Cholesky factorization of $A$. With (\ref{eq:tildea}),\[
L^{-1}\tilde{A}L^{-T}=I+L^{-1}(\tilde{A}-A)L^{-T}=I+L^{-1}EL^{-T},
\]
According to Theorem \ref{thm:err}, $L^{-1}EL^{-T}$ is positive semidefinite.
Thus, $\lambda(L^{-1}\tilde{A}L^{-T})\geq1$.

Theorem \ref{thm:err} also yields\begin{align*}
\Vert L^{-1}EL^{-T}\Vert_{2}  &  \leq\Vert E\Vert_{2}\Vert L^{-1}\Vert
_{2}\Vert L^{-T}\Vert_{2}\\
&  \leq\left[  (1+\tau^{2})^{l}-1\right]  \Vert A\Vert_{2}\Vert A^{-1}\Vert_{2}=\epsilon.
\end{align*}
Therefore,\[
1\leq\lambda(L^{-1}\tilde{A}L^{-T})\leq1+\epsilon.
\]
Since the eigenvalues of $\tilde{L}^{-1}A\tilde{L}^{-T}$ are the inverses of
those of $L^{-1}\tilde{A}L^{-T}$, we get (\ref{eq:eig2}).
\end{proof}

A comparison of the multilevel eSIF\ and SIF\ preconditioners is given in
Table \ref{tab:compare2}. The multilevel eSIF\ preconditioner has several
significant advantages over the SIF\ one.

\begin{enumerate}
\item The multilevel eSIF\ preconditioner is unconditionally robust or is
guaranteed to be positive definite, while the SIF\ one needs a strict (or even
impractical) condition to ensure the positive definiteness of the
approximation. That is, the SIF one needs $\hat{\epsilon}\equiv\left[
(1+\tau)^{l}-1\right]  \kappa(A)<1$. This means $\tau$ needs to be small
and/or the magnitudes of $l$ and $\kappa(A)$ cannot be very large.

\item The eSIF one gives a more accurate approximation to $A$ with a relative error bound $(1+\tau^{2})^{l}-1$ instead of $(1+\tau)^{l}-1$.

\item The eSIF one produces a much better condition number for the
preconditioned matrix ($1+\epsilon$ vs. $\frac{1+\hat{\epsilon}}{1-\hat{\epsilon}}$ with $\epsilon$ further much smaller than $\hat{\epsilon}$).

\item The eSIF one further produces better eigenvalue clustering for the
preconditioned matrix. The eigenvalues of the preconditioned matrix from eSIF
lie in $[\frac{1}{1+\epsilon},1]$, while those from SIF lie in a much larger
interval $[\frac{1}{1+\hat{\epsilon}},\frac{1}{1-\hat{\epsilon}}]$.
\end{enumerate}

A combination of these advantages makes the eSIF\ preconditioner much more
effective, as demonstrated later in numerical tests. \begin{table}[ptbh]
\caption{Comparison of $l$-level SIF and eSIF\ preconditioners that are used
to produce $\tilde{L}$ and $\tilde{A}$, where the results for the SIF
preconditioner are from \cite{hssprec}.}\label{tab:compare2}
\begin{center}
\tabcolsep5pt\renewcommand{\arraystretch}{1.6}
\begin{tabular}
[c]{c|c|c}\hline
& SIF & eSIF\\\hline
Existence/ & Conditional & \multirow{2}{*}{Unconditional}\\[-2mm]Positive definiteness & ($\hat{\epsilon}\equiv\left[  (1+\tau)^{l}-1\right]
\kappa(A)<1$) & \\
$\frac{\Vert\tilde{A}-A\Vert_{2}}{\Vert A\Vert_{2}}$ & $\leq(1+\tau)^{l}-1$ &
$\leq(1+\tau^{2})^{l}-1$\\
$\lambda(\tilde{L}^{-1}A\tilde{L}^{-T})$ & $\in\lbrack\frac{1}{1+\hat
{\epsilon}},\frac{1}{1-\hat{\epsilon}}]$ & $\in\lbrack\frac{1}{1+\epsilon},1]$\\
$\Vert\tilde{L}^{-1}A\tilde{L}^{-T}-I\Vert_{2}$ & $\leq\frac{\hat{\epsilon}}{1-\hat{\epsilon}}$ & $\leq\frac{\epsilon}{1+\epsilon}$\\
$\kappa(\tilde{L}^{-1}A\tilde{L}^{-T})$ & $\leq\frac{1+\hat{\epsilon}}{1-\hat{\epsilon}}$ & $\leq1+\epsilon$\\\hline
\end{tabular}
\end{center}
\end{table}

\section{Practical eSIF($l$) scheme\label{sec:alg}}

In our discussions above, some steps are used for convenience and are not
efficient for practical preconditioning. In the design of a practical scheme
for eSIF($l$), we need to take care of the following points.

\begin{enumerate}
\item Avoid expensive dense Cholesky factorizations like in (\ref{eq:d2tilde}).

\item Avoid the explicit formation of $C$ in (\ref{eq:c}) (needed in
(\ref{eq:l0})) which is too costly.

\item Compute the low-rank approximation of $C$ without the explicit form of
$C$.
\end{enumerate}

For the first point, we can let $Q$ be an orthogonal matrix extended from $V_{1}$ in (\ref{eq:svd}) so that
\[
Q^{T}V_{1}=\left(
\begin{array}
[c]{c}I\\
0
\end{array}
\right)  .
\]
Since $V_{1}$ has column size $r$ which is typically small for the purpose of
preconditioning, $Q$ can be conveniently obtained with the aid of $r$
Householder vectors. Due to this, $Q$ is generally different from $V$ in
(\ref{eq:svdc}). Then (\ref{eq:d2tilde}) can be replaced by\[
I-V_{1}\Sigma_{1}^{2}V_{1}^{T}=Q(I-\operatorname{diag}(\Sigma_{1}^{2},0))Q^{T}.
\]
Accordingly, $\tilde{A}$ in (\ref{eq:tildea0}) can be rewritten as\[
\tilde{A}=\left(  \!\begin{array}
[c]{cc}L_{1} & \\
L_{2}C^{T} & L_{2}\end{array}
\!\right)  \left(  \!\begin{array}
[c]{cc}I & \\
& Q
\end{array}
\!\right)  \left(  \!\begin{array}
[c]{cc}I & \\
& I-\operatorname{diag}(\Sigma_{1}^{2},0)
\end{array}
\!\right)  \left(  \!\begin{array}
[c]{cc}I & \\
& Q^{T}\end{array}
\!\right)  \left(  \!\begin{array}
[c]{cc}L_{1}^{T} & CL_{2}^{T}\\
& L_{2}^{T}\end{array}
\!\right)  .
\]
Thus, we can let\begin{align}
\tilde{L}  &  =\left(  \!\begin{array}
[c]{cc}L_{1} & \\
& L_{2}\end{array}
\!\right)  \left(  \!\begin{array}
[c]{cc}I & \\
C^{T} & I
\end{array}
\!\right)  \left(  \!\begin{array}
[c]{cc}I & \\
& Q\tilde{\Sigma}_{1}\end{array}
\!\right)  ,\quad\text{with}\label{eq:ltilde}\\
\tilde{\Sigma}_{1}  &  =\operatorname*{diag}(\sqrt{1-\sigma_{1}^{2}},\ldots,\sqrt{1-\sigma_{r}^{2}},1,\ldots,1),\nonumber
\end{align}
so that (\ref{eq:prec0}) still holds.

Next, we try to avoid the explicit formation of $C$ in (\ref{eq:c}) which is
too expensive. Note (\ref{eq:ltilde}) means
\[
\tilde{L}^{-1}=\left(  \!\begin{array}
[c]{cc}I & \\
& \tilde{\Sigma}_{1}^{-1}Q^{T}\end{array}
\!\right)  \left(  \!\begin{array}
[c]{cc}I & \\
-C^{T} & I
\end{array}
\!\right)  \left(  \!\begin{array}
[c]{cc}L_{1}^{-1} & \\
& L_{2}^{-1}\end{array}
\!\right)  .
\]
If $C$ is not formed but kept as the form in (\ref{eq:c}), then the
application of $\tilde{L}^{-1}$ to a vector involves four smaller solution
steps: one application of $L_{1}^{-1}$ to a vector, one application of
$L_{1}^{-T}$ to a vector, and two applications of $L_{2}^{-1}$ to vectors. To reduce the number of such
solutions, we rewrite $\tilde{L}$ in (\ref{eq:ltilde}) as
\begin{align}
\tilde{L}  &  =\left(  \!\!\!\begin{array}
[c]{cc}L_{1} & \\
& I
\end{array}
\!\!\!\right)  \left(  \!\!\!\begin{array}
[c]{cc}I & \\
A_{12}^{T}L_{1}^{-T} & L_{2}\end{array}
\!\!\!\right)  \left(  \!\!\!\begin{array}
[c]{cc}I & \\
& Q\tilde{\Sigma}_{1}\end{array}
\!\!\!\right) \label{eq:l1}\\
&  =\left(  \!\!\!\begin{array}
[c]{cc}L_{1} & \\
& I
\end{array}
\!\!\!\right)  \left(  \!\!\!\begin{array}
[c]{cc}I & \\
A_{12}^{T}L_{1}^{-T} & I
\end{array}
\!\!\!\right)  \left(  \!\!\!\begin{array}
[c]{cc}I & \\
& L_{2}Q\tilde{\Sigma}_{1}\end{array}
\!\!\!\right)  .\nonumber
\end{align}
$\tilde{L}^{-1}$ now has the following form and can be conveniently applied to
a vector:\begin{equation}
\tilde{L}^{-1}=\left(
\begin{array}
[c]{cc}I & \\
& \tilde{\Sigma}_{1}^{-1}Q^{T}L_{2}^{-1}\end{array}
\right)  \left(
\begin{array}
[c]{cc}I & \\
-A_{12}^{T}L_{1}^{-T} & I
\end{array}
\right)  \left(
\begin{array}
[c]{cc}L_{1}^{-1} & \\
& I
\end{array}
\right)  . \label{eq:invl1}\end{equation}
In fact, the application of $\tilde{L}^{-1}$ to a vector now just needs the
applications of $L_{1}^{-1}$, $L_{1}^{-T}$, $L_{2}^{-1}$ to vectors. In the
eSIF($l$) scheme, $L_{1}$ and $L_{2}$ are further approximated by structured
factors from the eSIF($l-1$) scheme. In addition, $Q^{T}$ is a Householder
matrix defined by $r$ Householder vectors and can be quickly applied to a
vector. $A_{12}^{T}$ is just part of $A$. With (\ref{eq:l1}), there is no need
to form $C$ explicitly. From these discussions, it is also clear how
$\tilde{L}^{-1}$ can be applied to vectors in actual preconditioning as
structured solution.

\begin{remark}
With the form of $\tilde{L}$ in (\ref{eq:l1}), it is clear that
(\ref{eq:prec0}) still holds for $\tilde{A}$ in (\ref{eq:tildea0}). Thus, the
approximation error result (\ref{eq:e}) in Theorem \ref{thm:err1} and the effectiveness results in
Theorem \ref{thm:ev1} remain the same. This further means that Theorems
\ref{thm:err} and \ref{thm:ev} for the multilevel scheme still hold.
\end{remark}

Thirdly, although $C$ needs not to be formed, it still needs to be compressed
so as to produce $\tilde{\Sigma}_{1}$ and $Q$ in (\ref{eq:l1}). We use
randomized SVD \cite{lib07} that is based on matrix-vector products. That is,
let\begin{equation}
Y=C^{T}Z=L_{2}^{-1}(A_{12}^{T}(L_{1}^{-1}Z)), \label{eq:cz}\end{equation}
where $Z$ is an appropriate skinny random matrix with column size $r+\alpha$
and $\alpha$ is a small constant oversampling size. $Y$ can be used to extract
$V_{1}$ approximately. After this, let\begin{equation}
T=CV_{1}=L_{1}^{-1}(A_{12}(L_{2}^{-T}V_{1})). \label{eq:t}\end{equation}
$TV_{1}^{T}$ essentially provides a low-rank approximation to $C$. Many
studies of randomized SVDs in recent years have shown the reliability of this
process. The tall and skinny matrix $T$ can then be used to quickly extract
$r$ approximate leading singular values of $C$. Accordingly, this process
provides an efficient way to get approximate $Q$ and $\tilde{\Sigma}_{1}$.

Computing $Y$ in (\ref{eq:cz}) and $T$ in (\ref{eq:t}) uses linear solves in
terms of $L_{1}$ and $L_{2}$ and matrix-vector multiplications in terms of
$A_{12}$. When $\tilde{L}$ results from the eSIF($l$) scheme, $L_{1}$ and
$L_{2}$ are approximated by structured eSIF($l-1$) factors.

We then study the costs to construct and apply the eSIF($l$) factor $\tilde
{L}$ and the storage of $\tilde{L}$. In practical, we specify $r$ instead of
$\tau$ in singular value truncations so as to explicitly control the cost. In
the following estimates, the precise leading term is given for the application
cost since it impacts the preconditioning cost.

\begin{proposition}
\label{prop:cost}Suppose $A$ is repeatedly bipartitioned into $l=\left\lfloor
\log N\right\rfloor $ levels with the diagonal blocks at each partition level
having the same size (for convenience). Let $\xi_{f}$ be the complexity to
compute the eSIF($l$) factor $\tilde{L}$ where each intermediate compression
step like (\ref{eq:svd}) uses rank $r$. Let $\xi_{s}$ be the complexity to apply $\tilde{L}^{-1}$ to a vector. Then
\[
\xi_{f}=O(rN^{2}),\quad\xi_{s}=2N^{2}+O(rN^{\log3}).
\]
The storage of $\tilde{L}$ is\[
\theta=O(rN\log N),
\]
excluding any storage for the blocks of $A$.
\end{proposition}

\begin{proof}
Let $\tilde{L}_{1}$ and $\tilde{L}_{2}$ be the eSIF($l-1$) factors that
approximate $L_{1}$ and $L_{2}$, respectively. For the eSIF($l$) factor
$\tilde{L}$, we use $\xi_{s}(N)$ to denote the cost to apply $\tilde{L}^{-1}$
to a vector. According to (\ref{eq:invl1}),
\[
\xi_{s}(N)=3\xi_{s}(\frac{N}{2})+2(\frac{N}{2})^{2}+O(r\frac{N}{2}),
\]
where the first term on the right-hand side is for applying $\tilde{L}_{1}^{-1}$, $\tilde{L}_{1}^{-T}$, $\tilde{L}_{2}^{-1}$ to vectors, the second
term is the dominant cost for multiplying $A_{12}^{T}$ in (\ref{eq:invl1}) to
a vector, and the third term is for the remaining costs (mainly to multiple
$Q^{T}$ to a vector). This gives a recursive relationship which can be expanded to yield
\begin{align*}
\xi_{s}(N)  &  =\frac{2}{3}N^{2}\sum_{i=1}^{l}\frac{3^{i}}{4^{i}}+O(rN\sum_{i=1}^{l}\frac{3^{i}}{2^{i}})\\
&  =2N^{2}+O(r3^{l})=2N^{2}+O(rN^{\log3}).
\end{align*}

Then consider the cost $\xi_{f}(N)$ to compute $\tilde{L}$. We have
\[
\xi_{f}(N)=2\xi_{f}(\frac{N}{2})+4(r+\alpha)\xi_{s}(\frac{N}{2})+4(r+\alpha
)(\frac{N}{2})^{2}+O(r^{2}\frac{N}{2}),
\]
where the first term on the right-hand side is for constructing the
eSIF($l-1$) factors $\tilde{L}_{1},\tilde{L}_{2}$, the second term is for
applying the relevant inverses of these factors as in (\ref{eq:cz}) and
(\ref{eq:t}) during the randomized SVD (with $\alpha$ the small constant
oversampling size), the third term is for multiplying $A_{12}^{T}$ and
$A_{12}$ to vectors as in (\ref{eq:cz}) and (\ref{eq:t}), and the last term is
for the remaining costs. Based on this recursive relationship, we can
similarly obtain $\xi_{f}(N)=O(rN^{2})$.

Finally, the storage $\theta(N)$ for $\tilde{L}$ (excluding the blocks of $A$)
mainly includes the storage for $\tilde{L}_{1},\tilde{L}_{2}$ and the $r$
Householder vectors for $Q$ in (\ref{eq:l1}):\[
\theta(N)=2\theta(\frac{N}{2})+O(rN).
\]
At the finest level of the partitioning of $A$, it also needs the storage of
$O(rN)$ for the Cholesky factors of the small diagonal blocks. Essentially,
the actual storage at each level is then $O(rN)$ and the total storage is
$\theta=O(rlN)$.
\end{proof}

We can see that the storage for the structured factors is roughly linear in
$N$ since $r$ is often fixed to be a small constant in preconditioning. The
cost to construct the preconditioner is just a one-time cost. The cost of
applying the preconditioner has a leading term $2N^{2}$. However, note that it
costs about $2N^{2}$ to multiply $A$ with a vector in each iteration anyway.
In the SIF case in \cite{hssprec}, it also costs $O(rN^{2})$ to construct the
multilevel preconditioner. The SIF application cost is lower but each
iteration step still costs $O(N^{2})$ due to the matrix-vector multiplication.
Furthermore, SIF\ preconditioners may not exist for some cases due to the loss
of positive definiteness. In the next section, we can see that the multilevel
eSIF\ preconditioner can often dramatically reduce the number of iterations so that it saves the total cost significantly.

\section{Numerical experiments\label{sec:tests}}

We then show the performance of the multilevel eSIF preconditioner in
accelerating the convergence of the preconditioned conjugate gradient method
(PCG). We compare the following three preconditioners.

\begin{itemize}
\item \textsf{bdiag}: the block diagonal preconditioner.

\item \textsf{SIF}: an SIF\ preconditioner from \cite{hssprec} (two versions
of SIF\ preconditioners are given in \cite{hssprec}, we use the one with
better robustness).

\item \textsf{eSIF}: the multilevel eSIF preconditioner.
\end{itemize}

In \cite{hssprec}, it has been shown that \textsf{SIF} is generally much more
effective than a preconditioner based on direct
rank-structured\ approximations. Here, we would like to show how \textsf{eSIF}
further outperforms \textsf{SIF}. The following notation is used to simplify
the presentation of the test results.

\begin{itemize}
\item $\gamma=\frac{\Vert Ax-b\Vert_{2}}{\Vert b\Vert_{2}}$: 2-norm relative
residual for a numerical solution $x$, with $b$ generated using the exact
solution vector of all ones.

\item $n_{\mathrm{iter}}$: total number of iterations to reach a certain
accuracy for the relative residual.

\item $A_{\mathrm{prec}}$: matrix preconditioned by the factors from the
preconditioners (for example, $A_{\mathrm{prec}}=\tilde{L}^{-1}A\tilde{L}^{-T}$ in the eSIF case).

\item $r$: numerical rank used in any low-rank approximation step in
constructing \textsf{SIF} and \textsf{eSIF}.

\item $l$: total number of levels in \textsf{SIF} and \textsf{eSIF}.
\end{itemize}

When \textsf{SIF} and \textsf{eSIF} are constructed, we use the same
parameters $r$, $l$, and finest level diagonal block size. The preconditioner
\textsf{bdiag} is constructed with the same diagonal block sizes as those of
the finest level diagonal block sizes of \textsf{SIF} and \textsf{eSIF}.

\begin{example}
\label{ex1}We first test the methods on the matrix $A$ with the $(i,j)$ entry\[
A_{ij}=\frac{(ij)^{1/4}\pi}{20+0.8(i-j)^{2}},
\]
which is modified from a test example in \cite{hssprec} to make it more challenging.
\end{example}

In the construction of \textsf{SIF} and \textsf{eSIF}, we use $r=5$. With the
matrix size $N$ increases, $l$ increases accordingly for \textsf{SIF} and
\textsf{eSIF} so that the finest level diagonal block size is fixed. Table
\ref{tab:ex1} shows the results of PCG iterations to reach the tolerance
$10^{-12}$ for the relative residual $\gamma$. Both \textsf{SIF} and
\textsf{eSIF} help significantly reduce the condition numbers. The both make
PCG converge much faster than using \textsf{bdiag}. \textsf{eSIF} is further
much more effective than \textsf{SIF} and leads to $\kappa(A_{\mathrm{prec}})$
close to $1$. PCG with \textsf{eSIF} only needs few steps to reach the desired
accuracy. The numbers of iterations are lower than with \textsf{SIF} by about
12 to 15 times.\begin{table}[ptbh]
\caption{\emph{Example \ref{ex1}}. Convergence results of PCG with
\textsf{bdiag}, \textsf{SIF}, and \textsf{eSIF} preconditioners. (For the two
largest matrices, it is very slow to form $A_{\mathrm{prec}}$, so the
condition numbers are not computed.)}\label{tab:ex1}
\begin{center}
\tabcolsep2pt\renewcommand{\arraystretch}{1.1}
\begin{tabular}
[c]{|c|c|c|c|c|c|c|c|}\hline
\multicolumn{2}{|c|}{$N$} & $1280$ & $2560$ & $5120$ & $10,240$ & $20,480$ &
$40,960$\\\hline
\multicolumn{2}{|c|}{$l$} & $8$ & $9$ & $10$ & $11$ & $12$ & $13$\\\hline
\multicolumn{2}{|c|}{$\kappa(A)$} & $2.66e7$ & $3.85e7$ & $5.55e7$ & $7.95e7$
&  & \\\hline\hline
\multirow{3}{*}{$\kappa(A_{\rm prec})$} & \textsf{bdiag} & $1.41e5$ & $1.42e5$
& $1.42e5$ & $1.42e5$ &  & \\
& \textsf{SIF} & $5.03e1$ & $5.03e1$ & $5.03e1$ & $5.03e1$ &  & \\
& \textsf{eSIF} & $1.01$ & $1.01$ & $1.02$ & $1.02$ &  & \\\hline
\multirow{3}{*}{$n_{\mathrm{iter}}$} & \textsf{bdiag} & $570$ & $562$ & $546$
& $551$ & $526$ & $525$\\
& \textsf{SIF} & $57$ & $60$ & $61$ & $60$ & $60$ & $60$\\
& \textsf{eSIF} & $4$ & $4$ & $4$ & $4$ & $4$ & $5$\\\hline
\multirow{3}{*}{$\gamma$} & \textsf{bdiag} & $9.65e\!-\!13$ & $9.49e\!-\!13$ &
$9.50e\!-\!13$ & $6.33e\!-\!13$ & $7.89e\!-\!13$ & $7.93e\!-\!13$\\
& \textsf{SIF} & $8.02e\!-\!13$ & $8.42e\!-\!13$ & $3.54e\!-\!13$ &
$9.36e\!-\!13$ & $7.36e\!-\!13$ & $8.28e\!-\!13$\\
& \textsf{eSIF} & $5.90e\!-\!15$ & $5.48e\!-\!15$ & $1.34e\!-\!13$ &
$4.28e\!-\!13$ & $5.00e\!-\!14$ & $9.61e\!-\!15$\\\hline
\end{tabular}
\end{center}
\end{table}

Figure \ref{fig:ex1conv}(a) shows the actual convergence behaviors for one
matrix and Figure \ref{fig:ex1conv}(b) reflects how the preconditioners change
the eigenvalue distributions. With \textsf{eSIF}, the eigenvalues of
$A_{\mathrm{prec}}$ are all closely clustered around $1$.

\begin{figure}[ptbh]
\centering\begin{tabular}
[c]{cc}\includegraphics[height=1.6in]{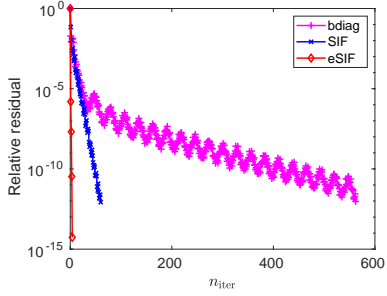} &
\includegraphics[height=1.6in]{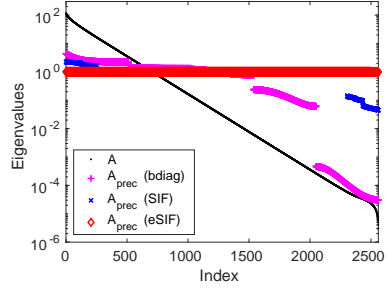}\\
{\scriptsize (a) Convergence} & {\scriptsize (b) Eigenvalues}\end{tabular}
\caption{\emph{Example \ref{ex1}}. Convergence of PCG with \textsf{bdiag},
\textsf{SIF}, and \textsf{eSIF} preconditioners and eigenvalues of the
preconditioned matrices for $N=2560$ in Table \ref{tab:ex1}.}\label{fig:ex1conv}\end{figure}

To confirm the efficiency of \textsf{eSIF}, we plot the storage requirement of
\textsf{eSIF} and the cost to apply the preconditioner in each step. Since $r$
is fixed, the storage of \textsf{eSIF} is $O(N\log N)$, which is confirmed in
Figure \ref{fig:ex1mem}.
\begin{figure}[!h]
\centering\begin{tabular}
[c]{cc}
\includegraphics[height=1.6in]{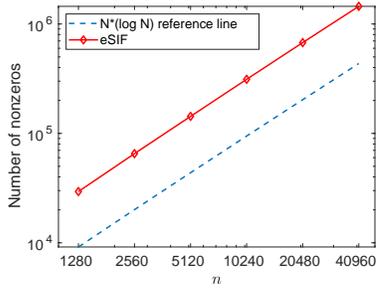} &
\includegraphics[height=1.6in]{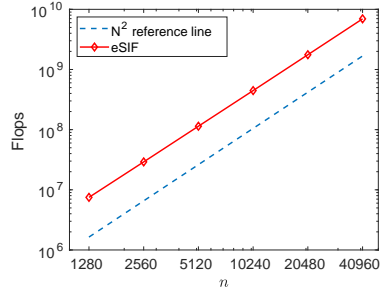}\\
{\scriptsize(a) Storage} & {\scriptsize(b) Application cost}
\end{tabular}
\caption{\emph{Example \ref{ex1}}. Storage for the structured factors of the
eSIF preconditioner (excluding the storage for $A$) and the application cost.}
\label{fig:ex1mem}
\end{figure}

\begin{example}
\label{ex2}In the second example, we consider to precondition some RBF (radial
basis function) interpolation matrices which are known to be notoriously
challenging for iterative methods due to the ill condition with some shape
parameters (see, e.g., \cite{boy11}). We consider the following four types of
RBFs:
\[
e^{-\varepsilon^{2}t^{2}},\quad\operatorname*{sech}\varepsilon t,\quad\frac
{1}{\sqrt{1+\varepsilon^{2}t^{2}}},\quad\frac{1}{1+\varepsilon^{2}t^{2}},
\]
where $\varepsilon$ is the shape parameter. The interpolation matrices are
obtained with grid points $0,1,\ldots,N-1$.
\end{example}

We test the RBF interpolation matrices $A$ with various different shape
parameters. With $N=1280$, $r=6$, and $l=8$, the performance of PCG\ to reach
the tolerance $10^{-12}$ for $\gamma$ is given in Table \ref{tab:ex2r6}. When
the shape parameter $\varepsilon$ reduces, the condition numbers of the
interpolation matrices increase quickly. \textsf{SIF} improves the condition
numbers more significantly than \textsf{bdiag}. However, for smaller
$\varepsilon$, the condition numbers resulting from both \textsf{bdiag} and
\textsf{SIF} get much worse and the convergence of PCG slows down.

\begin{table}[h]
\caption{\emph{Example \ref{ex2}}. Convergence results of PCG using
\textsf{bdiag}, \textsf{SIF}, and \textsf{eSIF} preconditioners with $r=6$ in
\textsf{SIF} and \textsf{eSIF}.}\label{tab:ex2r6}
\begin{center}
\tabcolsep2.2pt\renewcommand{\arraystretch}{1.1}
\begin{tabular}
[c]{|c|c|c|c|c|c|c|c|}\hline
\multicolumn{2}{|c|}{RBF} & \multicolumn{3}{c|}{$e^{-\varepsilon^{2}t^{2}}$} &
\multicolumn{3}{c|}{$\operatorname*{sech}\varepsilon t$}\\\hline
\multicolumn{2}{|c|}{$\varepsilon$} & $0.4$ & $0.36$ & $0.32$ & $0.3$ & $0.25$
& $0.2$\\\hline
\multicolumn{2}{|c|}{$\kappa(A)$} & $2.49e6$ & $9.27e7$ & $1.46e10$ & $3.48e6$
& $9.34e7$ & $1.30e10$\\\hline
\multirow{3}{*}{$\kappa(A_{\rm prec})$} & \textsf{bdiag} & $1.26e5$ & $4.50e6$
& $7.11e8$ & $1.52e5$ & $4.24e6$ & $6.28e8$\\
& \textsf{SIF} & $2.38$ & $2.11e3$ & $2.14e6$ & $1.34$ & $5.02e2$ & $7.58e5$\\
& \textsf{eSIF} & $1.00$ & $1.00$ & $1.00$ & $1.00$ & $1.00$ & $1.30$\\\hline
\multirow{3}{*}{$n_{\mathrm{iter}}$} & \textsf{bdiag} & $700$ & $2193$ &
$4482$ & $547$ & $1271$ & $3211$\\
& \textsf{SIF} & $15$ & $107$ & $549$ & $9$ & $52$ & $282$\\
& \textsf{eSIF} & $1$ & $1$ & $2$ & $1$ & $1$ & $3$\\\hline
\multirow{3}{*}{$\gamma$} & \textsf{bdiag} & $8.82e\!-\!13$ & $8.62e\!-\!13$ &
$8.97e\!-\!13$ & $7.97e\!-\!13$ & $9.28e\!-\!13$ & $8.25e\!-\!13$\\
& \textsf{SIF} & $4.94e\!-\!13$ & $5.16e\!-\!13$ & $9.86e\!-\!13$ &
$4.02e\!-\!13$ & $9.44e\!-\!13$ & $9.91e\!-\!13$\\
& \textsf{eSIF} & $6.16e\!-\!16$ & $7.34e\!-\!15$ & $2.63e\!-\!16$ &
$6.96e\!-\!15$ & $1.85e\!-\!13$ & $4.91e\!-\!14$\\\hline\hline
\multicolumn{2}{|c|}{RBF} & \multicolumn{3}{c|}{$\frac{1}{\sqrt{1+\varepsilon
^{2}t^{2}}}$} & \multicolumn{3}{c|}{$\frac{1}{1+\varepsilon^{2}t^{2}}$}\\\hline
\multicolumn{2}{|c|}{$\varepsilon$} & $0.3$ & $0.25$ & $0.2$ & $1/4$ & $1/5$ &
$1/6$\\\hline
\multicolumn{2}{|c|}{$\kappa(A)$} & $2.64e5$ & $2.27e6$ & $5.62e7$ & $1.42e5$
& $3.29e6$ & $7.59e7$\\\hline
\multirow{3}{*}{$\kappa(A_{\rm prec})$} & \textsf{bdiag} & $1.15e4$ & $9.64e4$
& $2.40e6$ & $6.18e3$ & $1.41e5$ & $3.34e6$\\
& \textsf{SIF} & $1.74$ & $6.30$ & $2.22e2$ & $1.94$ & $2.66e1$ & $8.91e2$\\
& \textsf{eSIF} & $1.00$ & $1.00$ & $1.26$ & $1.00$ & $1.00$ & $1.03$\\\hline
\multirow{3}{*}{$n_{\mathrm{iter}}$} & \textsf{bdiag} & $195$ & $375$ & $937$
& $190$ & $541$ & $1222$\\
& \textsf{SIF} & $13$ & $27$ & $86$ & $14$ & $43$ & $104$\\
& \textsf{eSIF} & $3$ & $3$ & $6$ & $2$ & $3$ & $5$\\\hline
\multirow{3}{*}{$\gamma$} & \textsf{bdiag} & $9.21e\!-\!13$ & $7.19e\!-\!13$ &
$8.92e\!-\!13$ & $9.84e\!-\!13$ & $9.16e\!-\!13$ & $7.52e\!-\!13$\\
& \textsf{SIF} & $4.23e\!-\!13$ & $5.14e\!-\!13$ & $6.20e\!-\!13$ &
$2.72e\!-\!13$ & $7.15e\!-\!13$ & $1.95e\!-\!13$\\
& \textsf{eSIF} & $1.77e\!-\!15$ & $1.62e\!-\!15$ & $8.16e\!-\!15$ &
$2.36e\!-\!13$ & $5.58e\!-\!13$ & $2.05e\!-\!15$\\\hline
\end{tabular}
\end{center}
\end{table}

On the other hand, \textsf{eSIF} performs significantly better for all the
cases. Dramatic reductions in the numbers of iterations can be observed. In
Table \ref{tab:ex2r6}, the number of PCG iterations with \textsf{eSIF} is up
to $274$ times lower than with \textsf{SIF} and up to $2241$ times lower than
with \textsf{bdiag}. Overall, PCG with \textsf{eSIF} takes just few iterations
to reach the desired accuracy.

Figure \ref{fig:ex2conv}(a) shows the actual convergence behaviors for one
case and Figure \ref{fig:ex2conv}(b) illustrates how the preconditioners
improve the eigenvalue distribution. Again, the eigenvalue clustering with
\textsf{eSIF} is much better. \begin{figure}[ptbh]
\centering\begin{tabular}
[c]{cc}\includegraphics[height=1.6in]{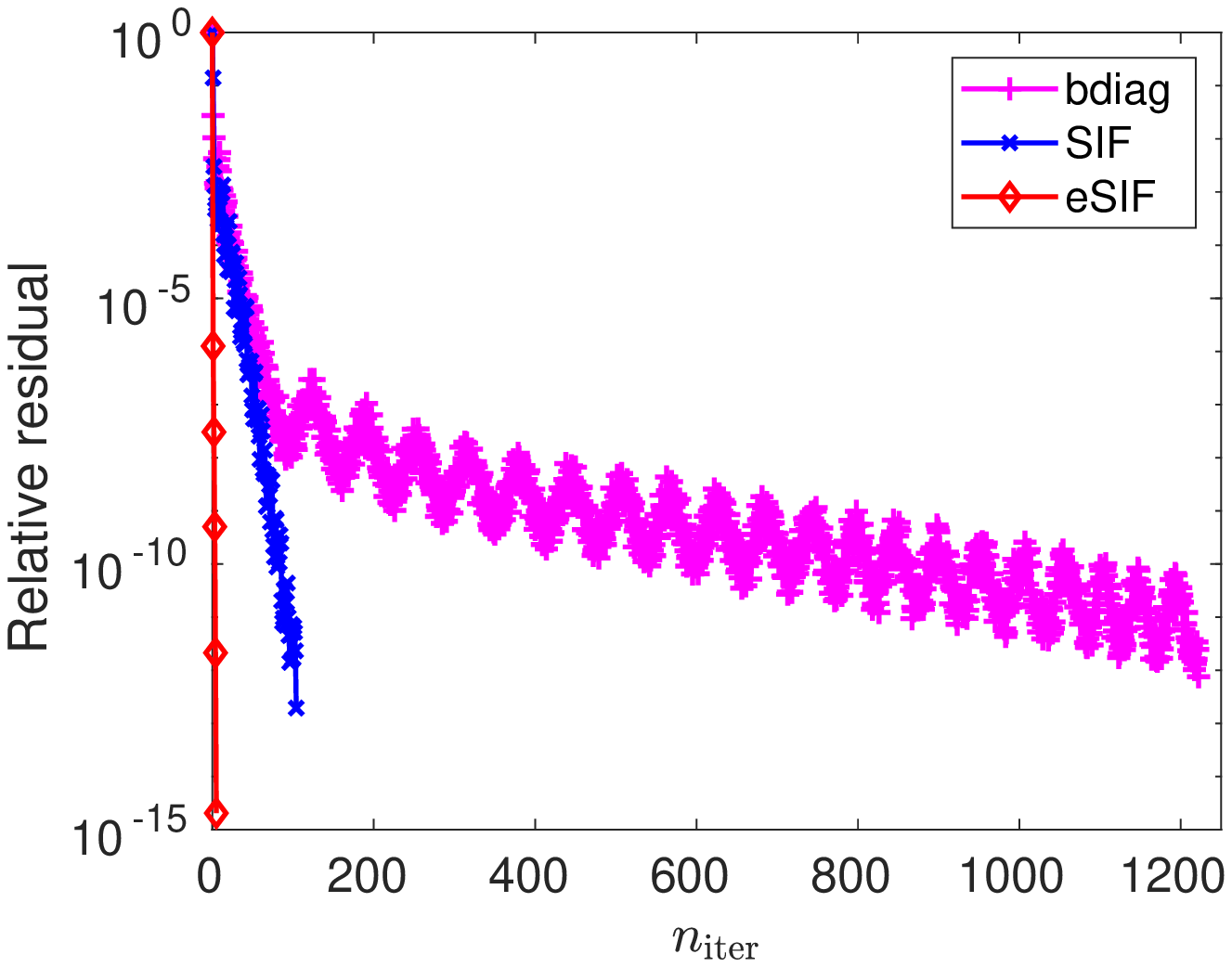} &
\includegraphics[height=1.6in]{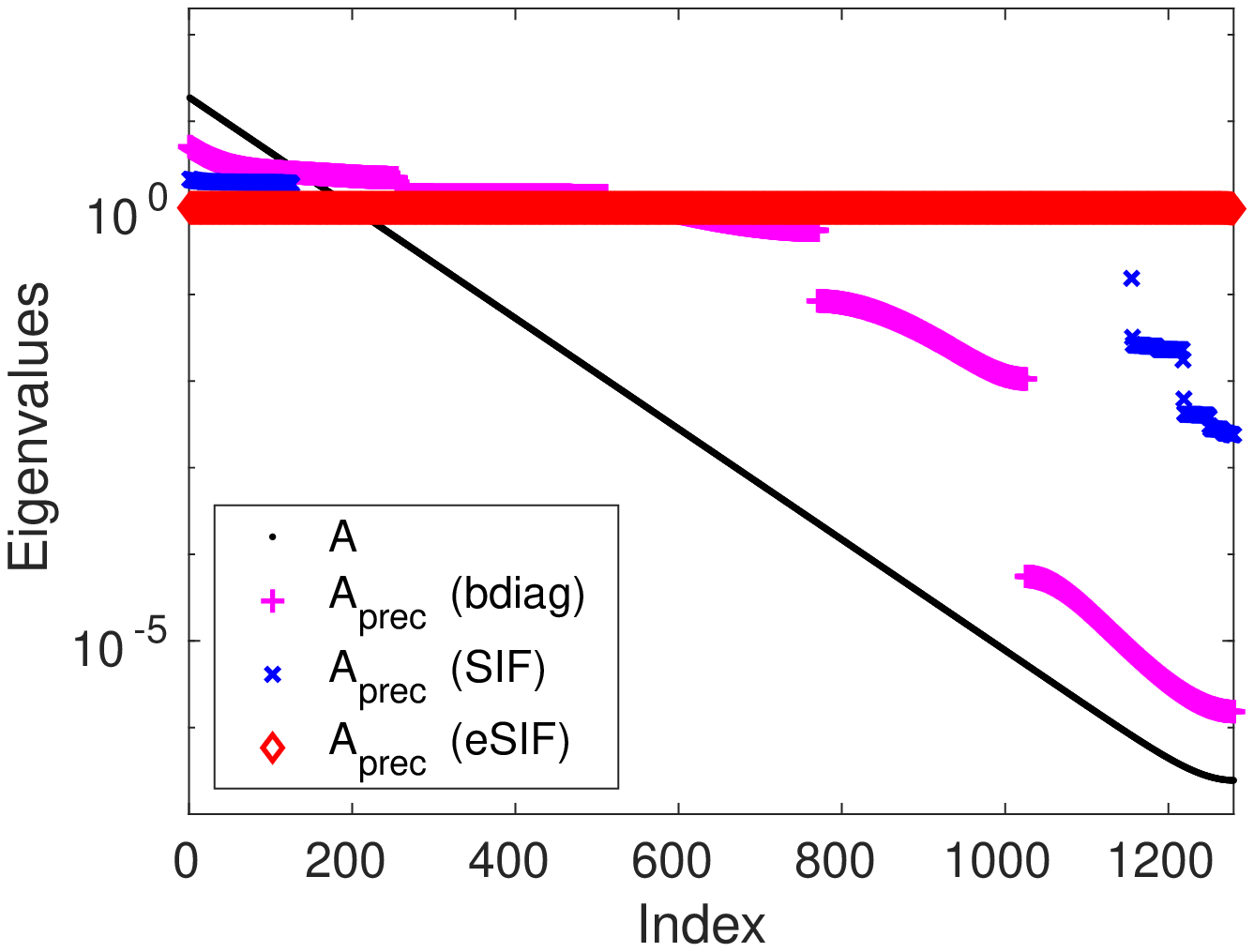}\\
{\scriptsize (a) Convergence} & {\scriptsize (b) Eigenvalues}\end{tabular}
\caption{\emph{Example \ref{ex2}}. Convergence of PCG and eigenvalues of the
preconditioned matrices for the case with RBF $\frac{1}{1+\varepsilon^{2}t^{2}}$, $\varepsilon=\frac{1}{6}$ in Table \ref{tab:ex2r6}.}\label{fig:ex2conv}\end{figure}

We also try different numerical ranks $r$ and the results are reported in
Table \ref{tab:ex2r84}. \textsf{SIF} is more sensitive to $r$. For some cases,
\textsf{SIF} with $r=4$ leads to quite slow convergence of PCG. In contrast,
\textsf{eSIF} remains very effective for the different $r$ choices and yields
much faster convergence.

\begin{table}[ptbh]
\caption{\emph{Example \ref{ex2}}. Convergence results of PCG using
\textsf{SIF} and \textsf{eSIF} preconditioners with different $r$.}\label{tab:ex2r84}
\begin{center}
\tabcolsep1.8pt\renewcommand{\arraystretch}{1.1}
\begin{tabular}
[c]{|c|c|c|c|c|c|c|c|c|}\hline
\multicolumn{3}{|c|}{RBF} & \multicolumn{3}{c|}{$e^{-\varepsilon^{2}t^{2}}$} &
\multicolumn{3}{c|}{$\operatorname*{sech}\varepsilon t$}\\\hline
\multicolumn{3}{|c|}{$\varepsilon$} & $0.3$ & $0.25$ & $0.2$ & $1/4$ & $1/5$ &
$1/6$\\\hline
\multirow{4}{*}{$\kappa(A_{\rm prec})$} & \multirow{2}{*}{$r=8$} &
\textsf{SIF} & $1.01$ & $2.35$ & $3.64e4$ & $1.00$ & $1.23$ & $4.80e3$\\
&  & \textsf{eSIF} & $1.00$ & $1.00$ & $1.00$ & $1.00$ & $1.00$ &
$1.06$\\\cline{2-9}
& \multirow{2}{*}{$r=4$} & \textsf{SIF} & $5.17e2$ & $7.51e4$ & $6.94e7$ &
$1.41e2$ & $4.61e4$ & $1.82e7$\\
&  & \textsf{eSIF} & $1.00$ & $1.00$ & $5.58$ & $1.00$ & $1.01$ &
$1.58e2$\\\hline
\multirow{4}{*}{$n_{\mathrm{iter}}$} & \multirow{2}{*}{$r=8$} & \textsf{SIF} &
$5$ & $13$ & $245$ & $4$ & $7$ & $69$\\
&  & \textsf{eSIF} & $1$ & $1$ & $1$ & $1$ & $1$ & $2$\\\cline{2-9}
& \multirow{2}{*}{$r=4$} & \textsf{SIF} & $178$ & $751$ & $3972$ & $92$ &
$410$ & $1613$\\
&  & \textsf{eSIF} & $2$ & $3$ & $17$ & $2$ & $3$ & $14$\\\hline
\multirow{4}{*}{$\gamma$} & \multirow{2}{*}{$r=8$} & \textsf{SIF} &
$7.95e\!-\!15$ & $2.90e\!-\!13$ & $4.95e\!-\!13$ & $2.64e\!-\!15$ &
$3.18e\!-\!13$ & $4.28e\!-\!13$\\
&  & \textsf{eSIF} & $6.89e\!-\!16$ & $1.08e\!-\!15$ & $1.23e\!-\!14$ &
$6.28e\!-\!15$ & $1.85e\!-\!13$ & $8.59e\!-\!13$\\\cline{2-9}
& \multirow{2}{*}{$r=4$} & \textsf{SIF} & $9.09e\!-\!13$ & $9.42e\!-\!13$ &
$4.36e\!-\!11$ & $8.11e\!-\!13$ & $6.92e\!-\!13$ & $6.06e\!-\!13$\\
&  & \textsf{eSIF} & $1.20e\!-\!15$ & $4.63e\!-\!15$ & $7.58e\!-\!13$ &
$9.14e\!-\!16$ & $8.64e\!-\!14$ & $6.33e\!-\!13$\\\hline\hline
\multicolumn{3}{|c|}{RBF} & \multicolumn{3}{c|}{$\frac{1}{\sqrt{1+\varepsilon
^{2}t^{2}}}$} & \multicolumn{3}{c|}{$\frac{1}{1+\varepsilon^{2}t^{2}}$}\\\hline
\multicolumn{3}{|c|}{$\varepsilon$} & $0.3$ & $0.25$ & $0.2$ & $1/4$ & $1/5$ &
$1/6$\\\hline
\multirow{4}{*}{$\kappa(A_{\rm prec})$} & \multirow{2}{*}{$r=8$} &
\textsf{SIF} & $1.39$ & $3.66$ & $1.06e2$ & $1.45$ & $6.32$ & $6.21e1$\\
&  & \textsf{eSIF} & $1.00$ & $1.00$ & $1.00$ & $1.00$ & $1.00$ &
$1.00$\\\cline{2-9}
& \multirow{2}{*}{$r=4$} & \textsf{SIF} & $6.96e1$ & $7.44e2$ & $2.47e4$ &
$2.98$ & $9.42e1$ & $1.91e4$\\
&  & \textsf{eSIF} & $1.03$ & $1.56$ & $1.18$ & $1.00$ & $1.06$ &
$4.34$\\\hline
\multirow{4}{*}{$n_{\mathrm{iter}}$} & \multirow{2}{*}{$r=8$} & \textsf{SIF} &
$10$ & $19$ & $75$ & $11$ & $27$ & $64$\\
&  & \textsf{eSIF} & $2$ & $2$ & $2$ & $2$ & $2$ & $3$\\\cline{2-9}
& \multirow{2}{*}{$r=4$} & \textsf{SIF} & $77$ & $224$ & $761$ & $19$ & $87$ &
$368$\\
&  & \textsf{eSIF} & $5$ & $8$ & $19$ & $4$ & $5$ & $14$\\\hline
\multirow{4}{*}{$\gamma$} & \multirow{2}{*}{$r=8$} & \textsf{SIF} &
$9.73e\!-\!14$ & $7.71e\!-\!13$ & $4.63e\!-\!13$ & $1.11e\!-\!13$ &
$2.50e\!-\!13$ & $6.97e\!-\!13$\\
&  & \textsf{eSIF} & $1.78e\!-\!15$ & $2.19e\!-\!14$ & $1.09e\!-\!13$ &
$1.44e\!-\!15$ & $3.02e\!-\!15$ & $1.95e\!-\!15$\\\cline{2-9}
& \multirow{2}{*}{$r=4$} & \textsf{SIF} & $5.93e\!-\!13$ & $9.84e\!-\!13$ &
$9.21e\!-\!13$ & $4.81e\!-\!13$ & $9.20e\!-\!13$ & $5.71e\!-\!13$\\
&  & \textsf{eSIF} & $8.38e\!-\!14$ & $9.19e\!-\!13$ & $1.87e\!-\!13$ &
$3.84e\!-\!15$ & $2.67e\!-\!13$ & $1.05e\!-\!13$\\\hline
\end{tabular}
\end{center}
\end{table}

\begin{example}
\label{ex3}In the last example, we compare \textsf{eSIF} with \textsf{SIF} in
terms of the following test matrices from different application backgrounds.

\begin{itemize}
\item \texttt{MHD3200B} ($N=3200$, $\kappa(A)=1.60e13$): The test matrix
MHD3200B from the Matrix Market \cite{matmart} treated as a dense matrix.
$r=9$ and $l=8$ are used in the test.

\item \texttt{ElasSchur} ($N=3198$, $\kappa(A)=8.91e6$): A Schur complement in
the factorization of a discretized linear elasticity equation as used in
\cite{schurmono}. The ratio of the so-called Lam\'{e} constants is $10^{5}$.
The original sparse discretized matrix has size $5,113,602$ and $A$
corresponds to the last separator in the nested dissection ordering of the
sparse matrix. $r=5$ and $l=9$ are used in the test.

\item \texttt{LinProg} ($N=2301$, $\kappa(A)=2.09e11$): A test example in
\cite{hssprec} from linear programming. The matrix is formed by $A=BDB^{T}$,
where $B$ is from the linear programming test matrix set Meszaros in
\cite{suitesparse} and $D$ is a diagonal matrix with diagonal entries evenly
located in $[10^{-5},1]$. $r=3$ and $l=9$ are used in the test.

\item \texttt{Gaussian} ($N=4000$, $\kappa(A)=1.41e10$): a matrix of the form
$sI+G$ with $G$ from the discretization of the Gaussian kernel $e^{-\frac
{\Vert t_{i}-t_{j}\Vert_{2}}{2\mu^{2}}}$. Such matrices frequently appear in
applications such as Gaussian processes. Here, $s=10^{-9}$, $\mu=2.5$ and the
$t_{i}$ points are random points distributed in a long three dimensional
rectangular parallelepiped. $r=20$ and $l=8$ are used in the test.
\end{itemize}
\end{example}

The convergence behaviors of PCG with \textsf{SIF} and \textsf{eSIF}
preconditioners are given in Figure \ref{fig:ex3}. Much faster convergence of
PCG can be observed with \textsf{eSIF}. For the four matrices listed in the
above order, the numbers of PCG iterations with \textsf{SIF} are about 11, 7,
7, and 21 times of those with \textsf{eSIF}, respectively.\begin{figure}[ptbh]
\centering\begin{tabular}
[c]{cc}\includegraphics[height=1.6in]{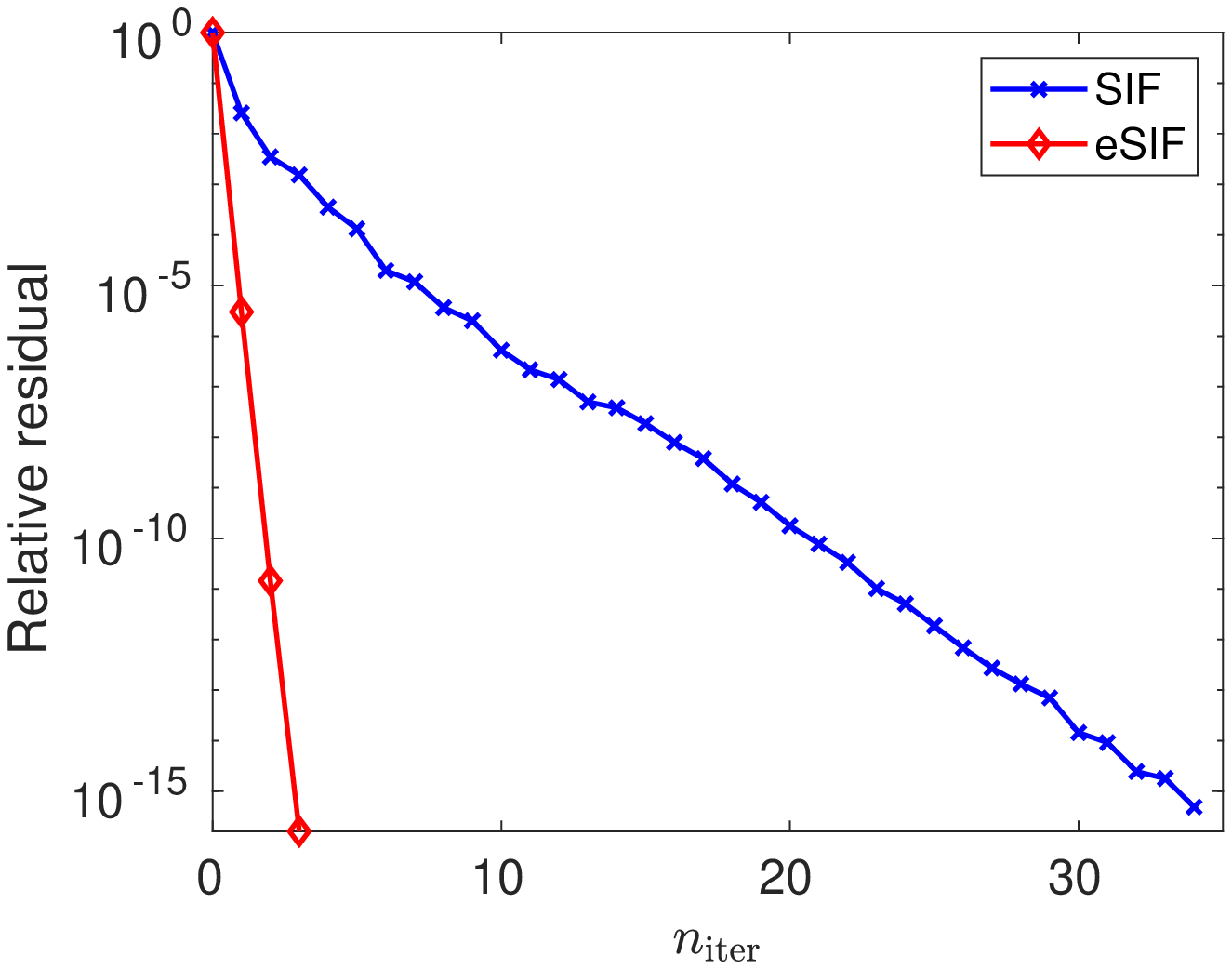} &
\includegraphics[height=1.6in]{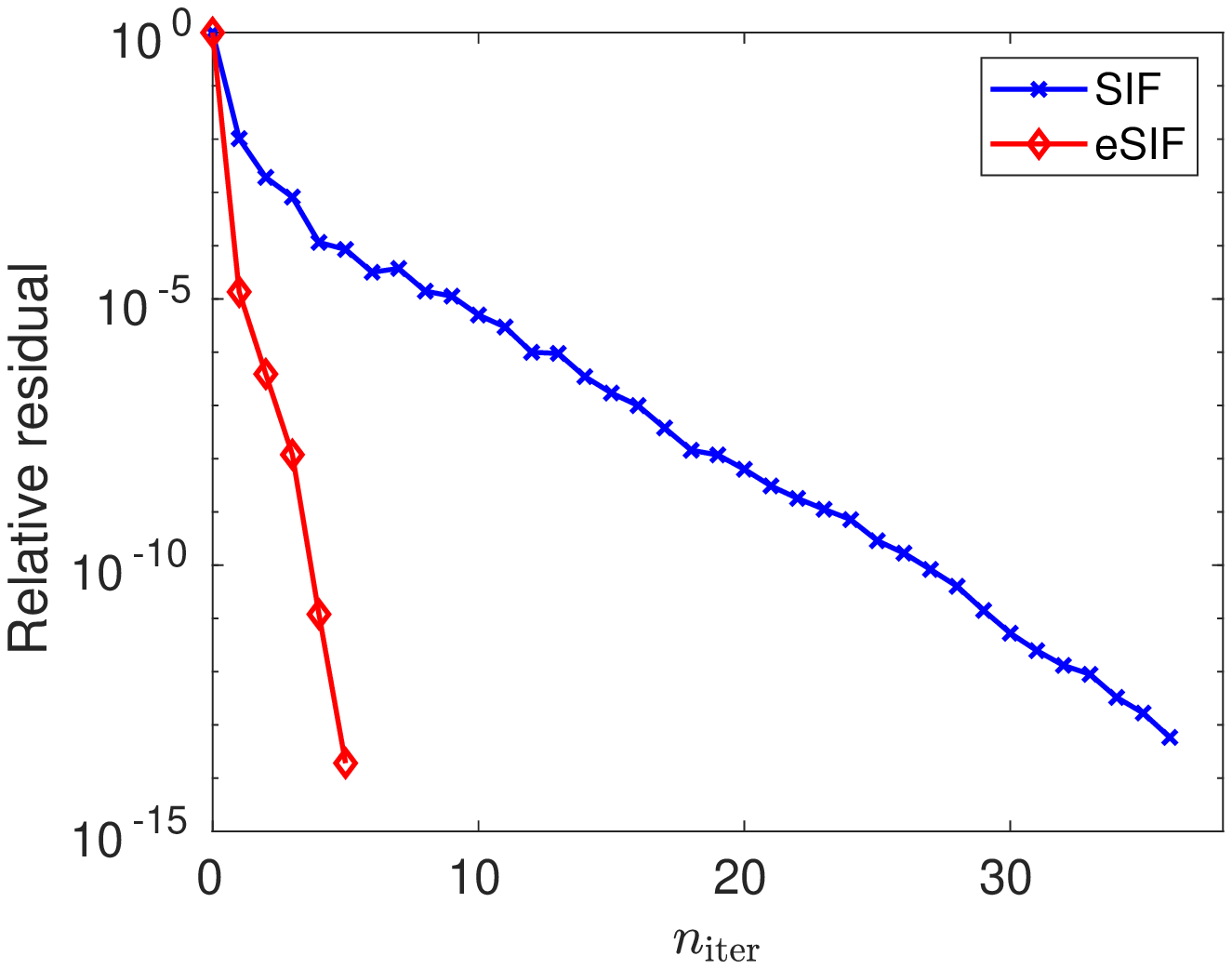}\\
{\scriptsize (a) \texttt{MHD3200B}} & {\scriptsize (b) \texttt{ElasSchur}}\\[2mm]\includegraphics[height=1.6in]{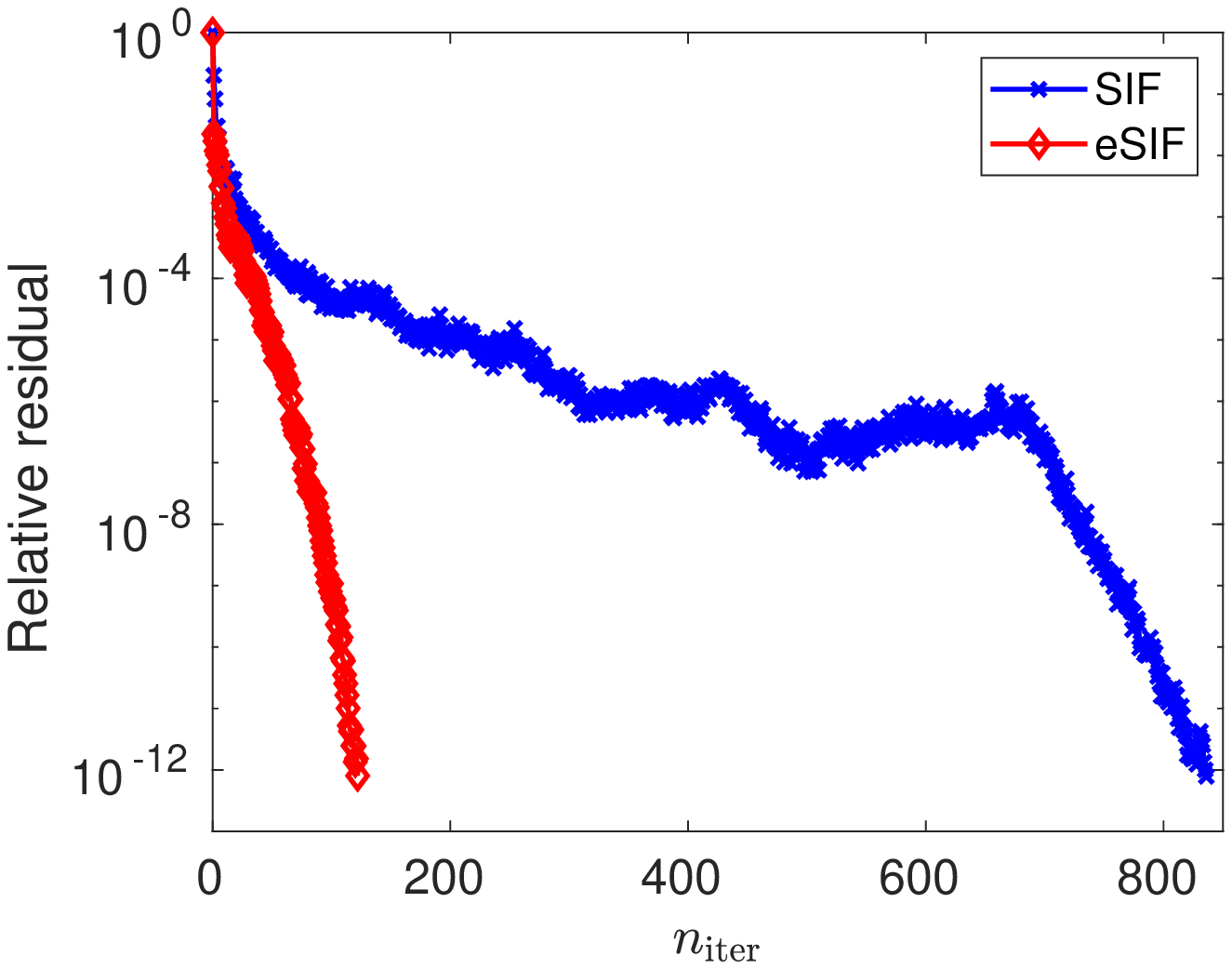} &
\includegraphics[height=1.6in]{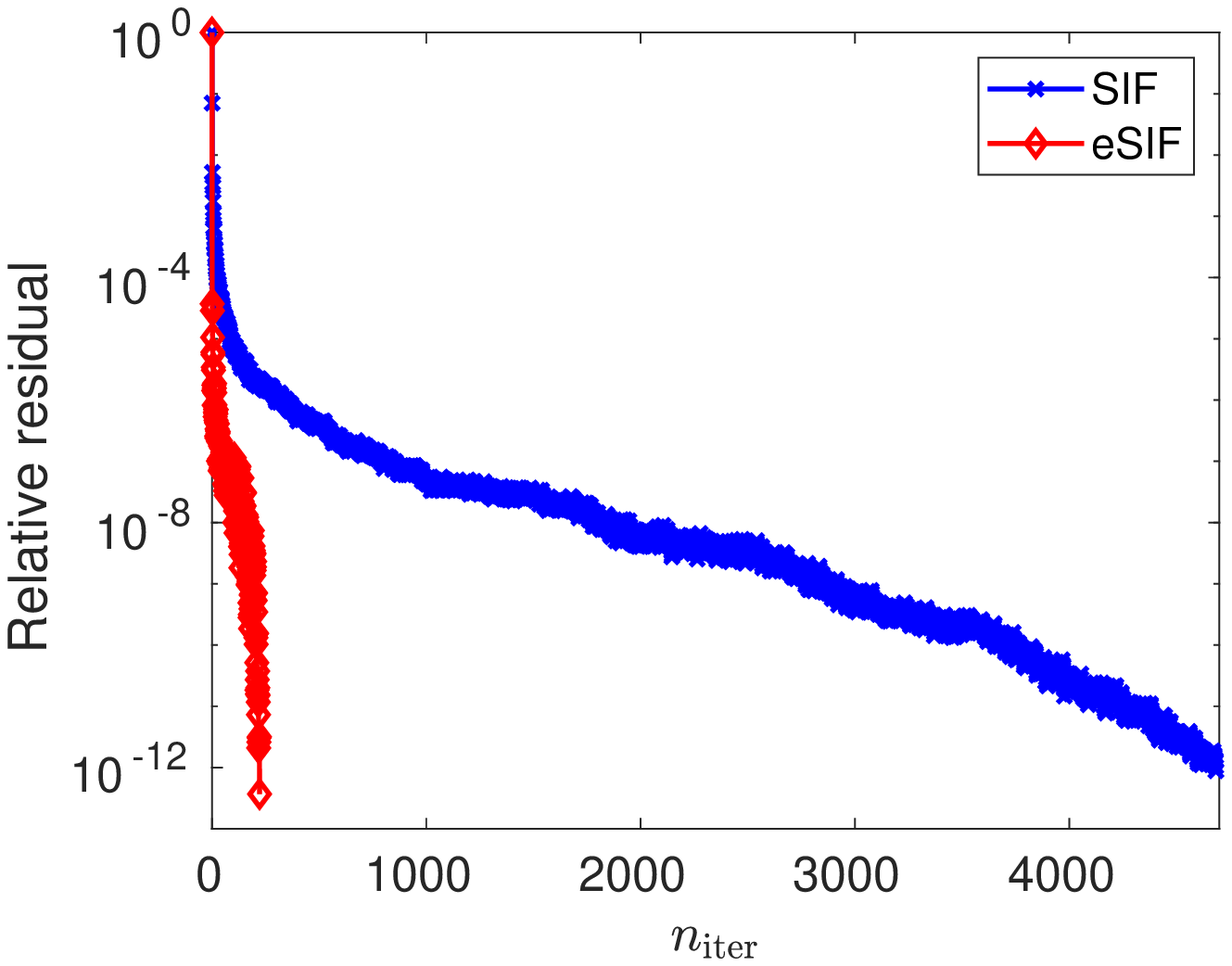}\\
{\scriptsize (c) \texttt{LinProg}} & {\scriptsize (d) \texttt{Gaussian}}\end{tabular}
\caption{\emph{Example \ref{ex3}}. Convergence of PCG with \textsf{SIF} and
\textsf{eSIF} preconditioners.}\label{fig:ex3}\end{figure}

\section{Conclusions\label{sec:concl}}

We have presented an \textsf{eSIF} framework that enhances a recent \textsf{SIF} preconditioner in multiple aspects.
During the construction of
the preconditioner, two-sided block triangular preprocessing is followed by
low-rank approximations in appropriate computations. Analysis of both the
prototype preconditioner and the practical multilevel extension is given. We
are able to not only overcome a major bottleneck of potential loss of positive definiteness in the
\textsf{SIF} scheme but also significantly improve the accuracy bounds,
condition numbers, and eigenvalue distributions. Thorough comparisons in terms
of the analysis and the test performance are given.

In our future work, we expect to explore new preprocessing and approximation
strategies that can further improve the eigenvalue clustering and accelerate
the decay magnification effect in the condition number. The current work
successfully improves the relevant accuracy, condition number, and eigenvalue
bounds by a significant amount (e.g., from $\frac{1+\hat{\epsilon}}{1-\hat{\epsilon}}$ to $1+\epsilon$ in Table \ref{tab:compare2} with
$\epsilon$ much smaller than $\hat{\epsilon}$). We expect to further continue this trend and in
the meantime keep the preconditioners convenient to apply. We will also
explore the feasibility of extending our ideas to nonsymmetric and indefinite matrices.

\end{document}